\documentclass[12pt,reqno]{amsart}

\usepackage[utf8]{inputenc}

\usepackage{amsmath} %General package for maths (loads also amsbsy to make bold symbols)
\usepackage{amsthm} %Package for theorems
\usepackage{amssymb} %Symbols (loads also amsfonts)
\usepackage{amscd} %Package for rectangular diagrams

\usepackage{euscript}
\DeclareMathAlphabet{\mathpzc}{OT1}{pzc}{m}{it}
\DeclareMathAlphabet\euscr{T1}{qzc}{m}{n}

\usepackage{emptypage}

\usepackage{enumitem} %Enumerations
\usepackage{mathtools} %More symbols, etc.
\usepackage{mathrsfs} %Calligraphic symbols with \mathscr
\usepackage{hyperref} %Makes references hyperlinks
\hypersetup{
    linktocpage = true,
    colorlinks=true,
    linkcolor=red,
}
\usepackage{esint} %To write averaged integrals

\usepackage{graphicx}
\usepackage{subcaption}

\usepackage{pgf,tikz,pgfplots}
\usetikzlibrary{arrows}
\usetikzlibrary{intersections}
\usetikzlibrary{fadings}

\newtheorem{theorem}{Theorem}[section]
\newtheorem*{theorem*}{Theorem}
\newtheorem{proposition}[theorem]{Proposition}
\newtheorem{lemma}[theorem]{Lemma}
\newtheorem{corollary}[theorem]{Corollary}

\theoremstyle{definition}
\newtheorem{definition}[theorem]{Definition}

\theoremstyle{remark}
\newtheorem{remark}[theorem]{Remark}
\newtheorem*{remark*}{Remark}

\newcommand{\con}[1]{\mathbb{#1}}
\newcommand{\R}{\con{R}} %Real
\newcommand{\N}{\con{N}} %Natural
\newcommand{\Sph}{\con{S}} %Sphere
\renewcommand{\H}{\con{H}}
\newcommand{\ccal}{\mathscr{C}}

\newcommand{\ical}{\mathcal{I}}
\newcommand{\lcal}{\mathcal{L}}
\newcommand{\ncal}{\mathcal{N}}
\newcommand{\ocal}{\mathcal{O}}

\newcommand{\e}{\mathrm{e}}

\makeatletter
\newcommand{\leqnomode}{\tagsleft@true\let\veqno\@@leqno}
\newcommand{\reqnomode}{\tagsleft@false\let\veqno\@@eqno}
\makeatother

\newcommand{\norm}[1]{\left \| {#1} \right \|}

\newcommand{\laplacian}{\Delta}
\newcommand{\s}{\gamma}
\newcommand{\fraclaplacian}{(-\Delta)^\s}

\renewcommand{\d}{\,\mathrm{d}} %straight d with small space before
\newcommand{\bpar}[1]{\left ( {#1}\right )}
\newcommand{\setcond}[2]{\left \{ #1 \ : \ #2  \right \}}

\newcommand{\average}{\fint}
\newcommand\beqc[1]{\left\{\begin{array}{#1}}
\newcommand\eeqc{\end{array} \right.}

\def\PDEsystem{rcll}
\def\bmatrix{\begin{pmatrix}}
\def\ematrix{\end{pmatrix}}

\DeclareMathOperator{\dist}{dist}

\DeclareMathOperator{\sign}{sign}

\newcommand{\usub}{\underline{u}}
\newcommand{\usup}{\overline{u}}
\newcommand{\vsub}{\underline{v}}
\newcommand{\vsup}{\overline{v}}
\def\ds{\displaystyle}

\numberwithin{equation}{section}

\graphicspath {{Images/}}

\title[Semilinear integro-differential equations II]{Semilinear integro-differential equations, II: one-dimensional and saddle-shaped solutions to the Allen-Cahn equation}
\author{Juan-Carlos Felipe-Navarro}
\address{J.C. Felipe-Navarro:
Universitat Polit\`ecnica de Catalunya and BGSMath, Departament de Matem\`{a}tiques, Diagonal 647, 08028 Barcelona, Spain}
\email{juan.carlos.felipe@upc.edu}

\author{Tomás Sanz-Perela}
\address{T. Sanz-Perela:
	Universitat Polit\`ecnica de Catalunya and BGSMath, Departament de Matem\`{a}tiques, Diagonal 647, 08028 Barcelona, Spain
	and
	School of Mathematics, The University of Edinburgh,
	Peter Tait Guthrie Road, EH9 3FD Edinburgh, UK}
\email{tomas.sanz-perela@ed.ac.uk}

\thanks{Both authors acknowledge financial support from the Spanish Ministry of Economy and Competitiveness (MINECO), through the María de Maeztu Programme for Units of Excellence in R\&D (MDM-2014-0445-16-4 and MDM-2014-0445 respectively). Moreover, both authors are supported by MINECO grants MTM2014-52402-C3-1-P and MTM2017-84214-C2-1-P, are members of the Barcelona Graduate School of Mathematics (BGSMath) and are part of the Catalan research group 2017 SGR 01392.}


\begin{document}

%%%%%%%%%%%%%%%%%%%%%%%%%%%%%%%%%%%%%%%%%%%%%%%%%%%%%%%%%%%%%%%%%%%%%%%%%%%%%%%
%%%%%%%%%%%%%%%%%%%%%%%%%%%%%%%%%%%%%%%%%%%%%%%%%%%%%%%%%%%%%%%%%%%%%%%%%%%%%%%
\begin{abstract}
This paper addresses saddle-shaped solutions to the semilinear equation $L_K u = f(u)$ in $\mathbb{R}^{2m}$, where $L_K$ is a linear elliptic integro-differential operator with a radially symmetric kernel $K$, and $f$ is of Allen-Cahn type. Saddle-shaped solutions are doubly radial, odd with respect to the Simons cone $\{(x', x'') \in \mathbb{R}^m \times \mathbb{R}^m \, : \, |x'| = |x''|\}$, and vanish only in this set.

We establish the uniqueness and the asymptotic behavior of the saddle-shaped solution. For this, we prove a Liouville type result, the one-dimensional symmetry of positive solutions to semilinear problems in a half-space, and maximum principles in ``narrow'' sets. The existence of the solution was already proved in part I of this work.
\end{abstract}
%%%%%%%%%%%%%%%%%%%%%%%%%%%%%%%%%%%%%%%%%%%%%%%%%%%%%%%%%%%%%%%%%%%%%%%%%%%%%%%
%%%%%%%%%%%%%%%%%%%%%%%%%%%%%%%%%%%%%%%%%%%%%%%%%%%%%%%%%%%%%%%%%%%%%%%%%%%%%%%

\maketitle

%%%%%%%%%%%%%%%%%%%%%%%
\section{Introduction}
%%%%%%%%%%%%%%%%%%%%%%%
\label{Sec:Introduction}

In this paper we study saddle-shaped solutions to the semilinear equation
\begin{equation}
\label{Eq:NonlocalAllenCahn}
L_K u = f(u) \quad \textrm{ in } \R^{2m},
\end{equation}
where $L_K$ is a linear integro-differential operator of the form \eqref{Eq:DefOfLu} and $f$ is of Allen-Cahn type. These  solutions (see Definition~\ref{Def:SaddleShapedSol} below) are particularly interesting in relation to the nonlocal version of a conjecture by De Giorgi, with the aim of finding a counterexample in high dimensions. Moreover, this problem is related to the regularity theory of nonlocal minimal surfaces (see Subsection~\ref{Subsec:DeGiorgi}).

Previous to this article and its first part \cite{FelipeSanz-Perela:IntegroDifferentialI}, there are only three works devoted to saddle-shaped solutions to the equation \eqref{Eq:NonlocalAllenCahn} with $L_K$ being the fractional Laplacian. In  \cite{Cinti-Saddle,Cinti-Saddle2}, Cinti proved the existence of a saddle-shaped solution as well as some qualitative properties, such as asymptotic behavior, monotonicity properties, and instability in even dimensions $2m\leq 6$. In a previous paper by the authors \cite{Felipe-Sanz-Perela:SaddleFractional}, further properties of these solutions were proved, the main ones being uniqueness and, when $2m\geq 14$, stability.
The possible stability in dimensions $8$, $10$, and $12$ is still an open problem, as well as the possible minimality of this solution in dimensions $2m \geq 8$.
Concerning saddle-shaped solutions to the classical Allen-Cahn equation $-\laplacian u = f(u)$, the same results are established; see \cite{Cabre-Saddle} and the references therein. The stability of the saddle-shaped solution to $-\laplacian u = u - u^3$ in dimensions $n= 8$, $10$, and $12$ has been recently announced \cite{LiuWangWei-StabilitySaddle}.

The present paper together with its first part \cite{FelipeSanz-Perela:IntegroDifferentialI} are the first ones in the literature studying saddle-shaped solutions for general integro-differential equations of the form \eqref{Eq:NonlocalAllenCahn}. 
In the three previous papers \cite{Cinti-Saddle, Cinti-Saddle2, Felipe-Sanz-Perela:SaddleFractional}, the extension problem for the fractional Laplacian was a key tool. 
Since this technique cannot be carried out for general integro-differential operators, some purely nonlocal techniques were developed in \cite{FelipeSanz-Perela:IntegroDifferentialI} and we exploit them in this article.

In part~I, we established an appropriate setting to study solutions to \eqref{Eq:NonlocalAllenCahn} that are doubly radial and odd with respect to the Simons cone, a property that is satisfied by saddle-shaped solutions (see Subsection~\ref{Subsec:Integro-differential setting}). 
We found an alternative and useful expression for the operator $L_K$ when acting on doubly radial odd functions ---see \eqref{Eq:OperatorOddF}. 
This was used to establish maximum principles for odd functions under a convexity assumption on the kernel $K$ of the operator $L_K$ ---see \eqref{Eq:SqrtConvex}. 
Moreover, we proved an energy estimate for doubly radial and odd minimizers of the energy associated to the equation, as well as the existence of saddle-shaped solutions to \eqref{Eq:NonlocalAllenCahn}.

In the current paper, we further study saddle-shaped solutions to \eqref{Eq:NonlocalAllenCahn}, by proving their uniqueness and asymptotic behavior.
To establish the uniqueness (Theorem~\ref{Th:ExistenceUniqueness}) we use a maximum principle for the linearized operator $L_K - f'(u)$ (Proposition~\ref{Prop:MaximumPrincipleLinearized}).
To prove the asymptotic behavior (Theorem~\ref{Th:AsymptoticBehaviorSaddleSolution}), we use two ingredients: a Liouville type theorem (Theorem~\ref{Th:LiouvilleSemilinearWholeSpace}) and a one-dimensional symmetry result (Theorem~\ref{Th:SymmHalfSpace}), both for semilinear equations of the form \eqref{Eq:NonlocalAllenCahn} under some hypotheses on $f$. 
The first of these results is obtained by adapting the ideas of Berestycki, Hamel, and Nadirashvili \cite{BerestyckiHamelNadi} to the nonlocal framework, and requires a Harnack inequality and a parabolic maximum principle. The second one requires the sliding method and the moving planes argument, extended to a general integro-differential setting. 

In addition to the previous results, in this paper we establish further properties of the so-called \emph{layer solution} $u_0$ (see Section~\ref{Sec:Asymptotic}). 
We also include an alternative proof of the existence of the saddle-shaped solution using the monotone iteration method (as an alternative to the proof in \cite{FelipeSanz-Perela:IntegroDifferentialI} where we used variational techniques).

Equation \eqref{Eq:NonlocalAllenCahn} is driven by a linear integro-differential operator $L_K$ of the form
\begin{equation}
\label{Eq:DefOfLu}
L_K w(x) = \int_{\R^n} \big(w(x) - w(y)\big) K(x-y)\d y.
\end{equation}
The most canonical example of such operators is the fractional Laplacian, which corresponds to the kernel $K(z) = c_{n, \s} |z|^{-n-2\s}$, where $\s \in (0,1)$ and $c_{n, \s}$ is a normalizing positive constant ---see \eqref{Eq:ConstantFracLaplacian}. Note that some of the results in this paper are new even for the fractional Laplacian (namely Proposition~\ref{Prop:MaximumPrincipleLinearized} and the statement on odd solutions of Theorem~\ref{Th:SymmHalfSpace}), while others are only proved in the literature using the extension problem (in contrast with our proofs).

Throughout the paper, we assume that $K$ is symmetric, i.e.,
\begin{equation}
\label{Eq:ParityK}
K(z) = K(-z),
\end{equation}
and that $L_K$ is uniformly elliptic, that is,
\begin{equation}
\label{Eq:Ellipticity}
\lambda \dfrac{c_{n,\s}}{|z|^{n+2\s}} \leq K(z) \leq \Lambda \dfrac{c_{n,\s}}{|z|^{n+2\s}}\,, 
\end{equation}
where $\lambda$ and $\Lambda$ are two positive constants. Conditions \eqref{Eq:ParityK} and \eqref{Eq:Ellipticity} are frequently adopted since they yield Hölder regularity of solutions (see \cite{RosOton-Survey,SerraC2s+alphaRegularity}). The family of linear operators satisfying these two conditions is the so-called $\lcal_0(n,\s,\lambda, \Lambda)$ ellipticity class. For short we will usually write $\lcal_0$ and we will make explicit the parameters only when needed. 

When dealing with doubly radial functions we will assume that the operator $L_K$ is rotation invariant, that is, $K$ is radially symmetric. This extra assumption allows us to rewrite the operator in a suitable form when acting on doubly radial odd functions, as explained next.

%%%%%%%%%%%%%%%%%%%%%%%%%%%%%%%%%%%%%%%%%%%%%%%%%%%%%%%%%%%%%%%%%%%%%%%%%%%%%%%%%%%%%%%%%%%%%%%%%%%%%%%%%%%%%%%%%%%%%%%%%%%%%%%%%%%%%%%%%%%%%%%%%%%%%%%%%%%%%%%%%%%%%%%%%%%%%%%%%%%%%%%%%%%%%%%%%%%%%%%%%%%%%%%%%%%%%%%%%%%%%%%%%%%%%%%%%%%%%%%%%%%%%%%%%%%%%%%%%%%%

\subsection{Integro-differential setting for odd functions with respect to the Simons cone}
\label{Subsec:Integro-differential setting}

In this subsection we present the basic definitions and terminology used along the paper.
We also recall the setting established in part~I~\cite{FelipeSanz-Perela:IntegroDifferentialI} to study the saddle-shaped solution (we refer to that article for more details). 

First, we present the Simons cone, a central object along this paper. It is defined in $\R^{2m}$ by
$$
\mathscr{C} := \setcond{x = (x', x'') \in \R^m \times \R^m = \R^{2m}}{|x'| = |x''|}\,.
$$
This cone is of importance in the theory of (local and nonlocal) minimal surfaces (see Subsection~\ref{Subsec:DeGiorgi}). We will use the letters $\ocal$ and $\ical$ to denote each of the parts in which $\R^{2m}$ is divided by the cone $\ccal$:
$$
\ocal:= \setcond{x = (x', x'') \in \R^{2m}}{|x'| > |x''|} \ \textrm{ and } \
\ical:= \setcond{x = (x', x'') \in \R^{2m}}{|x'| < |x''|}.
$$

Both $\ocal$ and $\ical$ belong to a family of sets in $\R^{2m}$ which are called of \emph{double revolution}. These are sets that are invariant under orthogonal transformations in the first $m$ variables, as well as under orthogonal transformations in the last $m$ variables. 
Related to this concept, we say that a function $w:\R^{2m}  \to \R$ is \emph{doubly radial} if it depends only on the modulus of the first $m$ variables and on the modulus of the last $m$ ones, i.e., $w(x) = w(|x'|,|x''|)$. 

We recall now the definition of $(\cdot)^\star$, an isometry that played a significant role in part~I. 
It is defined by
\begin{equation*}
\begin{matrix}
(\cdot)^\star \colon & \R^{2m}= \R^{m}\times \R^{m}  &\to&  \R^{2m}= \R^{m}\times \R^{m}  \\
& x = (x',x'') &\mapsto & x^\star = (x'',x')\,.
\end{matrix}
\end{equation*}
Note that this isometry is actually an involution that maps $\ocal$ into $\ical$ (and vice versa) and leaves the cone $\ccal$ invariant ---although not all points in $\ccal$ are fixed points of $(\cdot)^\star$, for instance, $x= (1, 0, \ldots, 0,1)$. Taking into account this transformation, we say that a doubly radial function $w$ is \emph{odd with respect to the Simons cone} if $w(x) = -w(x^\star)$.

Now we can define saddle-shaped solutions.
\begin{definition}
	\label{Def:SaddleShapedSol}
	We say that a bounded solution $u$ to \eqref{Eq:NonlocalAllenCahn} is a \emph{saddle-shaped solution} (or simply \emph{saddle solution}) if
	\begin{enumerate}
		\item $u$ is doubly radial.
		\item $u$ is odd with respect to the Simons cone.
		\item $u > 0$ in $\ocal = \{|x'| > |x''|\} $.
	\end{enumerate}
\end{definition}
Note that these solutions are even with respect to the coordinate axes and that their zero level set is the Simons cone $\mathscr{C} = \{|x'|=|x''|\}$.

In part~I, we developed a purely nonlocal theory regarding the integro-differential operator $L_K$ when acting on odd solutions with respect to the Simons cone.
First, recall that if $K$ is a radially symmetric kernel we can rewrite the operator  acting on a doubly radial function $w$ as
$$
L_K w(x) = \int_{\R^{2m}} \big(w(x) - w(y)\big) \overline{K}(x,y) \d y\,,
$$
where $\overline{K}$ is doubly radial in both variables and is defined by
\begin{equation*}
\overline{K}(x,y) := \average_{O(m)^2} K(|Rx - y|)\d R\,.
\end{equation*}
Here, $\d R$ denotes integration with respect to the Haar measure on $O(m)^2$, where $O(m)$ is the orthogonal group of $\R^m$. 
It is important to notice that, in contrast with $K=K(x-y)$, the kernel $\overline{K}$ is no longer translation invariant (i.e., it is a function of $x$ and $y$ but not of the difference $x-y$).

If we consider doubly radial functions that are, in addition, odd with respect to the Simons cone, we can use the involution $(\cdot)^\star$ to find that
\begin{equation}
\label{Eq:OperatorOddF}
L_K w (x) = \int_{\ocal} \big(w(x) - w(y) \big) \big(\overline{K}(x, y) - \overline{K}(x, y^\star)  \big) \d y +  2 w(x) \int_{\ocal} \overline{K}(x, y^\star) \d y \,.
\end{equation}
Furthermore,
\begin{equation}
\label{Eq:ZeroOrderTerm}
\frac{1}{C} \dist(x,\ccal)^{-2\s} \leq \int_{\ocal} \overline{K}(x, y^\star) \d y \leq C \dist(x,\ccal)^{-2\s},
\end{equation}
with $C>0$ depending only on $m, \s, \lambda$, and $\Lambda$.

Note that the expression \eqref{Eq:OperatorOddF} has an integro-differential part plus a term of order zero with a positive coefficient. 
Thus, the most natural assumption to make in order to have an elliptic operator (when acting on doubly radial odd functions) is that the kernel of the integro-differential term is positive. 
That is,
\begin{equation}
	\label{Eq:KernelInequality}
	\overline{K}(x,y) - \overline{K}(x, y^\star)  > 0 \quad \text{ for every }x,y \in \ocal\,.
\end{equation}
One of the main results in part~I established a necessary and sufficient condition on the original kernel $K$ for $L_K$ to have a positive kernel when acting on doubly radial odd functions. It turns to be
	\begin{equation}
	\label{Eq:SqrtConvex}	
	K(\sqrt{\tau}) \text{ is a strictly convex function of }\tau\,.
	\end{equation}

The positivity of the kernel of $L_K$ when acting on doubly-radial odd functions was crucial in order to obtain the existence of the saddle-shaped solution.
As we will see, it is essential as well to establish the uniqueness.
Therefore, \eqref{Eq:SqrtConvex} will be a key assumption in some of our results.

%%%%%%%%%%%%%%%%%%%%%%%%%%%%%%%%%%%%%%%%%%%%%%%%%%%%%%%%%%%%%%%%%%%%%%%%%%%%%%%%%%%%%%%%%%%%%%%%%%%%%%%%%%%%%%%%%%%%%%%%%%%%%%%%%%%%%%%%%%%%%%%%%%%%%%%%%%%%%%%%%%%%%%%%%%%%%%%%%%%%%%%%%%%%%%%%%%%%%%%%%%%%%%%%%%%%%%%%%%%%%%%%%%%%%%%%%%%%%%%%%%%%%%%%%%%%%%%%%%%%
\subsection{Main results}
\label{Subsec:Main results}

Through all the paper we will assume that $f$, the nonlinearity in \eqref{Eq:NonlocalAllenCahn}, is a $C^1$ function satisfying
\begin{equation}
\label{Eq:Hypothesesf}
f \textrm{ is odd, } \quad f(\pm 1)=0, \quad \text{ and } \quad f \textrm{ is strictly concave in }  (0,1).
\end{equation}
It is easy to see that these properties yield $f>0$ in $(0,1)$, $f'(0)>0$ and $f'(\pm 1) < 0$. 

In some statements in this article, we will denote by $L^1_\s(\R^n)$ the space of measurable functions $w$ satisfying
$$
\int_{\R^n} \dfrac{|w(x)|}{1+|x|^{n+2\s}}\d x < +\infty\,.
$$
This regularity will be required on a function $w$ (in addition to $C^\alpha$ Hölder continuity, with $\alpha > 2\s$) to ensure that $L_K w$ is well-defined.

The first main result of this paper concerns uniqueness of saddle-shaped solution.

\begin{theorem}
	\label{Th:ExistenceUniqueness}
	Let $f$ satisfy \eqref{Eq:Hypothesesf}. Let $K$ be a radially symmetric kernel satisfying the convexity assumption \eqref{Eq:SqrtConvex} and such that $L_K\in \lcal_0(2m, \s, \lambda, \Lambda)$. 
	
	Then, for every even dimension $2m \geq 2$, there exists a unique saddle-shaped solution $u$ to \eqref{Eq:NonlocalAllenCahn}. In addition, $u$ satisfies $|u|<1$ in $\R^{2m}$.
\end{theorem}

To establish the uniqueness of the saddle-shaped solution we will need two ingredients: the asymptotic behavior of saddle-shaped solutions and a maximum principle for the linearized operator in $\ocal$. 
Both results will be described below.
The existence of saddle-shaped solutions was already proved in part~I \cite{FelipeSanz-Perela:IntegroDifferentialI} using variational techniques. 
Here, we show that it can also be established using, instead, the monotone iteration procedure. 
Let us remark that, in both methods, having the convexity assumption \eqref{Eq:SqrtConvex} is crucial.

The second main result of this paper is Theorem~\ref{Th:AsymptoticBehaviorSaddleSolution} below, on the asymptotic behavior of a saddle-shaped solution at infinity. 
To state it, let us introduce an important type of solutions in the study of the integro-differential Allen-Cahn equation: the layer solutions.

We say that a solution $v$ to $L_K v = f(v)$ in $\R^n$ is a \emph{layer solution} if $v$ is increasing in one direction, say $e\in \Sph^{n-1}$, and $v(x) \to \pm 1$ as $x\cdot e \to \pm \infty$ (not necessarily uniform). 
When $n=1$, a result of Cozzi and Passalacqua (Theorem~1 in \cite{CozziPassalacqua}) establishes the existence and uniqueness (up to translations) of a layer solution. In addition, this solution is odd with respect to some point. They assume the kernel to be in the ellipticity class $\lcal_0(1,\s,\lambda, \Lambda)$ and the nonlinearity satisfying \eqref{Eq:Hypothesesf}. In the case of the fractional Laplacian this result was proved in \cite{CabreSolaMorales,CabreSireII} using the extension problem.

Given $K$ a translation invariant kernel in $\R^n$, we define a new kernel $K_1$ in $\R$ as
$$
K_1(\tau) := \int_{\R^{n-1}} K\left(\theta,\tau\right) \d \theta = |\tau|^{n-1} \int_{\R^{n-1}} K\left(\tau\sigma,\tau\right) \d \sigma.
$$ 
Then, we denote by $u_0$ the (unique) layer solution in $\R$ associated to $L_{K_1}$ that vanishes at the origin. 
That is, 
\begin{equation}
\label{Eq:LayerSolution}
\beqc{\PDEsystem}
L_{K_1}  u_0 &=& f(u_0) & \textrm{ in }\R\,,\\
\dot{u}_0 &>& 0 & \textrm{ in } \R\,,\\
u_0(x) & = &-u_0(-x)  & \textrm{ in }\R\,,\\
\ds \lim_{x \to \pm \infty} u_0(x) &=& \pm 1. & 
\eeqc
\end{equation}
This solution will play an important role to establish the asymptotic behavior of saddle-shaped solutions. 
Indeed, its importance lies in that the associated function
\begin{equation}
\label{Eq:DefOfU}
U(x):= u_0 \left( \dfrac{|x'| - |x''|}{\sqrt{2}} \right)\,
\end{equation}
will describe the asymptotic behavior of saddle solutions at infinity. Note that $(|x'| - |x''| )/\sqrt{2}$ is the signed distance to the Simons cone (see Lemma~4.2 in \cite{CabreTerraII}). Therefore, the function $U$ consists of ``copies'' of the layer solution $u_0$ centered at each point of the Simons cone and oriented in the normal direction to the cone.

The precise statement on the asymptotic behavior of saddle-shaped solutions at infinity is the following.

\begin{theorem}
	\label{Th:AsymptoticBehaviorSaddleSolution}
	Let $f\in C^2(\R)$ satisfy \eqref{Eq:Hypothesesf}. Let $K$ be a radially symmetric kernel satisfying the convexity assumption \eqref{Eq:SqrtConvex} and such that $L_K\in \lcal_0(2m, \s, \lambda, \Lambda)$. Let $u$ be a saddle-shaped solution to \eqref{Eq:NonlocalAllenCahn} and let $U$ be the function defined by \eqref{Eq:DefOfU}.
	
	Then,
	$$
	\norm{u-U}_{L^\infty(\R^n\setminus B_R)}
	+\norm{\nabla u-\nabla U}_{L^\infty(\R^n\setminus B_R)}
	+\norm{D^2u-D^2U}_{L^\infty(\R^n\setminus B_R)} \to 0
	$$
	as $ R\to +\infty$.
\end{theorem}

Let us now describe some of the main ingredients that are used to prove Theorems~\ref{Th:ExistenceUniqueness} and \ref{Th:AsymptoticBehaviorSaddleSolution}. Concerning the uniqueness of the saddle-shaped solution, besides the asymptotic behavior described in Theorem~\ref{Th:AsymptoticBehaviorSaddleSolution} we also need the following maximum principle in $\ocal$ for the linearized operator $L_K - f'(u)$.

\begin{proposition}
	\label{Prop:MaximumPrincipleLinearized}
	Let $\Omega \subset \ocal$ be an open set (not necessarily bounded) and let $K$ be a radially symmetric kernel satisfying the convexity assumption \eqref{Eq:SqrtConvex} and such that $L_K\in \lcal_0(2m, \s, \lambda, \Lambda)$. Let $u$ be a saddle-shaped solution to \eqref{Eq:NonlocalAllenCahn}, and let $v\in  L^1_\s(\R^{2m})$ be a doubly radial function which is $C^\alpha$ in $\Omega$ and continuous up to the boundary, for some $\alpha > 2\s$. Assume that $v$ satisfies
	$$
	\beqc{\PDEsystem}
	L_K v - f'(u)v - c(x)v &\leq & 0 &\textrm{ in } \Omega\,,\\
	v &\leq & 0 &\textrm{ in } \ocal \setminus \Omega\,,\\
	- v(x^\star) & = & v(x) &\textrm{ in } \R^{2m},\\
	\ds \limsup_{x\in \Omega, \ |x|\to \infty} v(x) &\leq & 0\,,
	\eeqc
	$$
	with $c\leq 0$ in $\Omega$.
	
	Then, $v \leq 0$ in $\Omega$.
\end{proposition}

To establish it, the key tool is to use a maximum principle in ``narrow'' sets, also proved in Section~\ref{Sec:MaximumPrinciple}. Our proof of this result is much simpler than that of the analogue maximum principle for the classical Laplacian. This is an example of how the nonlocality of the operator can make some arguments easier and less technical (informally speaking, the reason would be that $L_K$ ``sees more'', or ``further'', than the Laplacian). It is also interesting to notice that the proof of Proposition~\ref{Prop:MaximumPrincipleLinearized} is by far simpler than the one using the extension problem in the case of the fractional Laplacian (Proposition~1.4 in \cite{Felipe-Sanz-Perela:SaddleFractional}). In the proof, the positivity condition \eqref{Eq:KernelInequality} ---guaranteed by the convexity of the kernel--- is crucial, together with the bounds \eqref{Eq:ZeroOrderTerm}.

Regarding the proof of Theorem~\ref{Th:AsymptoticBehaviorSaddleSolution}, to establish the asymptotic behavior of saddle-shaped solutions we use a compactness argument as in \cite{CabreTerraII, Cinti-Saddle, Cinti-Saddle2}, together with two important results presented next and established in Section~\ref{Sec:SymmetryResults}. 
The first one, Theorem~\ref{Th:LiouvilleSemilinearWholeSpace}, is a Liouville type principle for nonnegative solutions to a semilinear equation in the whole space. This result, in contrast with the previous ones, does not require the kernel $K$ to be radially symmetric, but only to satisfy \eqref{Eq:ParityK} and \eqref{Eq:Ellipticity}.

\begin{theorem}
	\label{Th:LiouvilleSemilinearWholeSpace}
	Let $L_K \in \lcal_0(n,\s,\lambda, \Lambda)$ and let $v$ be a bounded solution to
	\begin{equation}
	\label{Eq:PositiveWholeSpace}
	\beqc{\PDEsystem}
	L_K v &=& f(v) & \textrm{ in }\R^n\,,\\
	v &\geq& 0 & \textrm{ in } \R^n\,,
	\eeqc
	\end{equation}
	with a nonlinearity $f\in C^1$ satisfying
	\begin{itemize}
		\item $f(0) = f(1) = 0$,
		\item $f'(0)>0$,
		\item $f>0$ in $(0,1)$, and $f<0$ in $(1,+\infty)$.
	\end{itemize}
	Then, $v\equiv 0$ or $v \equiv 1$.
\end{theorem}

Similar classification results have been proved for the fractional Laplacian in \cite{ChenLiZhang,LiZhang} (either using the extension problem or not) with the method of moving spheres, which uses crucially the scale invariance of the operator $\fraclaplacian$. To the best of our knowledge, there is no similar result available in the literature for general kernels in the ellipticity class $\lcal_0$ (which are not necessarily scale invariant). Thus, we present here a proof based on the techniques introduced by Berestycki, Hamel, and Nadirashvili \cite{BerestyckiHamelNadi} for the local equation with the classical Laplacian. It relies on a maximum principle for a nonlinear heat equation, the translation invariance of the operator, a Harnack inequality, and a stability argument.

The second ingredient needed to prove the asymptotic behavior of saddle-shaped solutions is a symmetry result for equations in a half-space, stated next. Here and in the rest of the paper we use the notation $\R^n_+= \{(x_H,x_n)\in \R^{n-1}\times \R \ : \ x_n > 0\}$.  

\begin{theorem}
	\label{Th:SymmHalfSpace}
	Let $L_K\in \lcal_0(n,\s,\lambda, \Lambda)$ and let $v$ be a bounded solution to one of the following two problems: either to	
	\begin{equation}
	\reqnomode
	\tag{P1}
	\label{Eq:P1}
	\beqc{\PDEsystem}
	L_K v &=& f(v)   &\textrm{ in } \,\R^n_+,\\
	v &>& 0   &\textrm{ in } \,\R^n_+,\\
	v(x_H,x_n) &=& -v(x_H,-x_n)   &\textrm{ in } \,\R^n,
	\eeqc
	\end{equation}
	or to
	\begin{equation}
	\reqnomode
	\tag{P2}
	\label{Eq:P2}
	\beqc{\PDEsystem}
	L_K v &=& f(v)   &\textrm{ in } \,\R^n_+,\\
	v &>& 0   &\textrm{ in } \,\R^n_+,\\
	v &=& 0   &\textrm{ in } \,\R^n \setminus \R^n_+.
	\eeqc
	\end{equation}
	
	\reqnomode
	
	Assume that, in $\R^n_+$, the kernel $K$ of the operator $L_K$ is decreasing in the direction of $x_n$, i.e., it satisfies
	$$
	K(x_H-y_H,x_n-y_n) \geq K(x_H-y_H,x_n+y_n) \,\,\,\,\text{for all } \,\, x,y\in \R^n_+.
	$$ 
	Suppose that $f\in C^1$ and
	\begin{itemize}
		\item $f(0) = f(1) = 0$,
		\item $f'(0)>0$, and $f'(\tau)\leq 0$ for all $\tau\in[1-\delta,1]$ for some $\delta>0$,
		\item $f>0$ in $(0,1)$, and
		\item $f$ is odd in the case of \eqref{Eq:P1}.
	\end{itemize}
	Then, $v$ depends only on $x_n$ and it is increasing in this direction.
\end{theorem}

The result for \eqref{Eq:P2} has been proved for the fractional Laplacian under some assumptions on $f$ (weaker than the ones in Theorem~\ref{Th:SymmHalfSpace}) in \cite{QuaasXia, DipierroSoaveValdinoci, BarriosEtAl-Monotonicity, BarriosEtAl-Symmetry, FallWethMonotonicity}. Instead, no result was available for general integro-differential operators. To the best of our knowledge, problem \eqref{Eq:P1} on odd solutions with respect to a hyperplane has not been treated even for the fractional Laplacian. In our case, the fact that $f$ is of Allen-Cahn type allows us to use rather simple arguments that work for both problems \eqref{Eq:P1} and \eqref{Eq:P2} ---moving planes and sliding methods, similarly as done in \cite{DipierroSoaveValdinoci}. Moreover, the fact that the kernel of the operator is $|\cdot|^{-n-2\s}$ or a general $K$ satisfying uniform ellipticity bounds does not affect significantly the proof.  Although \eqref{Eq:P2} will not be used in this paper, we include it here for future reference since the proof for this problem is analogous to the one for \eqref{Eq:P1}.

%%%%%%%%%%%%%%%%%%%%%%%%%%%%%%%%%%%%%%%%%%%%%%%%%%%%%%%%%%%%%%%%%%%%%%%%%%%%%%%%%%%%%%%%%%%%%%%%%%%%%%%%%%%%%%%%%%%%%%%%%%%%%%%%%%%%%%%%%%%%%%%%%%%%%%%%%%%%%%%%%%%%%%%%%%%%%%%%%%%%%%%%%%%%%%%%%%%%%%%%%%%%%%%%%%%%%%%%%%%%%%%%%%%%%%%%%%%%%%%%%%%%%%%%%%%%%%%%%%%%
\subsection{Saddle-shaped solutions in the context of a conjecture by De Giorgi and the theory of nonlocal minimal surfaces}
\label{Subsec:DeGiorgi}

To conclude this introduction, let us make some comments on the importance of problem \eqref{Eq:NonlocalAllenCahn} and its relation with the theory of (classical and nonlocal) minimal surfaces and a famous conjecture raised by De Giorgi.

A main open problem (even in the local case) is to determine whether the saddle-shaped solution is a minimizer of the energy functional associated to the equation, depending on the dimension $2m$. 
This question is deeply related to the regularity theory of local and nonlocal minimal surfaces, as explained next.

It is well-known that, for powers $\s \in [1/2, 1]$, the rescaled energy functionals associated to the equation $\fraclaplacian u = f(u)$ $\Gamma$-converge to the classical perimeter functional (see \cite{AlbertiBouchitteSeppecher,Gonzalez}), while in the case $\s \in (0,1/2)$, they $\Gamma$-converge to the \emph{fractional perimeter} functional (see \cite{SavinValdinoci-GammaConvergence}). 
Thus, a blow-down sequence of minimizers of the Allen-Cahn energy converges to the characteristic function of a set whose classical or fractional perimeter (depending on the power $\s$) is minimal.

In the recent years there has been an increasing interest in developing a regularity theory for nonlocal minimal surfaces, although very few results are known for the moment. 
It is beyond the scope of this article to describe all of them in detail, and we refer the interested reader to \cite{CozziFigalli-Survey, BucurValdinoci} and the references therein. 
Let us just make some comments on the scarce available results concerning the possible minimality of the Simons cone as a nonlocal minimal surface, since this is connected to our work on saddle-shaped solutions. 
Note first that, due to all its symmetries, it is easy to check that the Simons cone $\ccal$ is stationary for the fractional perimeter. 
If $2m=2$, a purely geometric argument shows that it cannot be a minimizer (see \cite{Valdinoci2013FractionalPerimeter}).
Note indeed that in \cite{SavinValdinoci-Cones} Savin and Valdinoci proved that all minimizing nonlocal minimal cones in $\R^2$ are flat, and that dimension $2$ is the only one where a complete classification of minimizing nonlocal minimal cones is available. 
In higher dimensions, the only available results regarding the possible minimality of $\ccal$ appear in \cite{DaviladelPinoWei} and in our paper \cite{Felipe-Sanz-Perela:SaddleFractional}, but they concern stability, a weaker property than minimality. 

In \cite{DaviladelPinoWei}, Dávila, del Pino, and Wei found a very interesting characterization of the stability of the Simons cone. It consists of an inequality involving  two hypergeometric constants which depend only on $\s$ and the dimension. This inequality is checked numerically in \cite{DaviladelPinoWei}, finding that, in dimensions $2m \leq 6$ and for $\s$ close to zero, the Simons cone is not stable. Numerics also show that the Simons cone should be stable in dimension $8$ if $\s$ is close to zero. These two facts for small $\s$ fit with the general belief that, in the fractional setting, the Simons cone should be stable (and even a minimizer) in dimensions $2m \geq 8$ (as in the local case), probably for all $\s\in(0,1/2)$, though this is still an open problem. 

In contrast with the numeric computations in \cite{DaviladelPinoWei}, our arguments in \cite{Felipe-Sanz-Perela:SaddleFractional} establishing the stability of $\ccal$ in dimensions $2m \geq 14$ are the first analytical proof of a stability result for the Simons cone in any dimension (in the nonlocal setting). Our approach, which is completely different from theirs, relies on establishing the stability of the saddle-shaped solution and using that this property is preserved along a blow-down limit. This shows that the saddle-shaped solution does not only have its interest in the context of the Allen-Cahn equation, but it can also provide strategies to prove stability and minimality results in the theory of nonlocal minimal surfaces.

In addition to all this, saddle-shaped solutions are natural objects to build a counterexample to a famous conjecture raised by De Giorgi, asking whether bounded monotone solutions to $-\Delta  u = u - u^3$ in $\R^n$ are one-dimensional if $n\leq8$.
This conjecture is still nowadays not completely closed (see \cite{FarinaValdinoci-DeGiorgi} and references therein), but a counterexample in dimensions $n\geq 9$ was given in \cite{delPinoKowalczykWei} by using the gluing method. 
An alternative approach to the one of \cite{delPinoKowalczykWei} to construct a counterexample to the conjecture was given by Jerison and Monneau in \cite{JerisonMonneau}. 
They showed that a counterexample in $\R^{n+1}$ can be constructed with a rather natural procedure if there exists a global minimizer of $-\Delta u = f(u)$ in $\R^n$ which is bounded and even with respect to each coordinate but is not one-dimensional. 
The saddle-shaped solution is of special interest in search of this counterexample, since it is even with respect to all the coordinate axis and it is canonically associated to the Simons cone, which in turn is the simplest nonplanar minimizing minimal surface. 
Therefore, by proving that the saddle solution to the classical Allen-Cahn equation is a minimizer in some dimension $2m$, one would obtain automatically a counterexample to the conjecture in $\R^{2m+1}$.

For a more complete account on the available results concerning the conjecture by De Giorgi in the nonlocal setting, as well as to related conjectures on minimizers and stable solutions (in which the saddle-shaped solution is expected to have a role as a counterexample), we refer the interested reader to \cite{SanzPerela-Thesis} and the references therein.

%%%%%%%%%%%%%%%%%%%%%%%%%%%%%%%%%%%%%%%%%%%%%%%%%%%%%%%%%%%%%%%%%%%%%%%%%%%%%%%%%%%%%%%%%%%%%%%%%%%%%%%%%%%%%%%%%%%%%%%%%%%%%%%%%%%%%%%%%%%%%%%%%%%%%%%%%%%%%%%%%%%%%%%%%%%%%%%%%%%%%%%%%%%%%%%%%%%%%%%%%%%%%%%%%%%%%%%%%%%%%%%%%%%%%%%%%%%%%%%%%%%%%%%%%%%%%%%%%%%%
\subsection{Plan of the article}
\label{Subsec:Plan}

The paper is organized as follows. In Section~\ref{Sec:Preliminaries} we present some preliminary results that will be used in the rest of the article.
Section~\ref{Sec:ExistenceUniqueness} contains the proof of the uniqueness of a saddle-shaped solution, as well as the alternative proof of existence ---via the monotone iteration method. 
In Section~\ref{Sec:SymmetryResults} we establish the Liouville type and symmetry results, Theorems~\ref{Th:LiouvilleSemilinearWholeSpace} and \ref{Th:SymmHalfSpace}. Section~\ref{Sec:Asymptotic} is devoted to the layer solution $u_0$ of problem \eqref{Eq:NonlocalAllenCahn}, and to the proof of the asymptotic behavior of saddle-shaped solutions, Theorem~\ref{Th:AsymptoticBehaviorSaddleSolution}. Finally, Section~\ref{Sec:MaximumPrinciple} concerns the proof of a maximum principle in $\ocal$ for the linearized operator $L_K - f'(u)$ (Proposition~\ref{Prop:MaximumPrincipleLinearized}).

%%%%%%%%%%%%%%%%%%%%%%%%
\section{Preliminaries}
%%%%%%%%%%%%%%%%%%%%%%%%
\label{Sec:Preliminaries}

In this section we collect some preliminary results that will be used in the rest of this paper. First, we summarize the regularity results needed in the forthcoming sections. Then, we state a remark on stability that will be used later in this paper, and finally we recall the basic maximum principles for doubly radial odd functions proved in \cite{FelipeSanz-Perela:IntegroDifferentialI}.

%%%%%%%%%%%%%%%%%%%%%%%%%%%%%%%%%%%%%%%%%%%%%%%%%%%%%%%%
\subsection{Regularity theory for nonlocal operators in the class $\lcal_0$}
\label{Subsec:Regularity}
%%%%%%%%%%%%%%%%%%%%%%%%%%%%%%%%%%%%%%%%%%%%%%%%%%%%%%%%

In this subsection we present the regularity results that will be used in the paper. For further details, see \cite{RosOton-Survey,SerraC2s+alphaRegularity,CozziPassalacqua} and the references therein.

We first give a result on the interior regularity for linear equations.

\begin{proposition}[\cite{RosOton-Survey,SerraC2s+alphaRegularity}]
	\label{Prop:InteriorRegularity}
	Let $L_K \in\lcal_0(n,\s,\lambda, \Lambda)$ and let $w\in L^\infty (\R^n)$ be a weak solution to $L_K w = h$ in $B_1$. Then,
	\begin{equation}
	\label{Eq:C2sEstimate}
	\norm{w}_{C^{2\s} (B_{1/2})} \leq C\bpar{\norm{h}_{L^\infty (B_1)} + \norm{w}_{L^\infty  (\R^n)} }.
	\end{equation}
	Moreover, let $\alpha > 0$ and assume additionally that $w \in C^\alpha (\R^n)$. Then, if $\alpha +
	2\s$ is not an integer,
	\begin{equation}
	\label{Eq:Calpha->Calpha+2sEstimate}
	\norm{w}_{C^{\alpha + 2\s} (B_{1/2})} \leq C\bpar{\norm{h}_{C^{\alpha} (B_1)} + \norm{w}_{C^\alpha (\R^n)} },
	\end{equation}
	where $C$ is a constant that depends only on $n$, $\s$, $\lambda$, and $\Lambda$.
\end{proposition}

Throughout the paper we consider $u$ to be a saddle solution to \eqref{Eq:NonlocalAllenCahn} that satisfies $|u|\leq 1$ in $\R^n$. Hence, by applying \eqref{Eq:C2sEstimate} we find that for any $x_0\in \R^n$,
\begin{align*}
\norm{u}_{C^{2\s} (B_{1/2} (x_0))} &\leq C\bpar{\norm{f(u)}_{L^\infty (B_1(x_0))} + \norm{u}_{L^\infty  (\R^n)} } \\
&\leq C\bpar{1 + \norm{f}_{L^\infty ([-1,1])} }.
\end{align*}
Note that the estimate is independent of the point $x_0$, and thus since the equation is satisfied in the whole $\R^n$,
$$
\norm{u}_{C^{2\s}(\R^n)} \leq C\bpar{1 + \norm{f}_{L^\infty ([-1,1])} }.
$$
Then, we use estimate \eqref{Eq:Calpha->Calpha+2sEstimate} repeatedly and the same kind of arguments yield that, if $f\in C^{k}([-1,1])$, then $u\in C^{\alpha}(\R^n)$ for all $\alpha < k+ 2 \s$. Moreover, the following estimate holds:
$$
\norm{u}_{C^{\alpha}(\R^n)} \leq C\,,
$$
for some constant $C$ depending only on $n$, $\s$, $\lambda$, $\Lambda$, $k$, and $\norm{f}_{C^k([-1,1])}$.

Let us now state a result on the boundary regularity of solutions to a Dirichlet problem for an operator $L_K\in\lcal_0$. 

\begin{proposition}[\cite{CozziPassalacqua,RosOton-Survey}]
	\label{Prop:BoundaryRegularity}
	Let $L_K \in\lcal_0(n,\s,\lambda, \Lambda)$ and let $w\in L^\infty (\R^n)$ be a weak solution to
	$$
	\beqc{\PDEsystem}
	L_K w & = & h & \text{ in } \Omega\,,\\
	w & = & \varphi & \text{ in } \R^n \setminus \Omega\,,
	\eeqc
	$$
	with $h\in L^\infty(\Omega)$ and $\varphi \in C^{2\s + \eta}(\R^n\setminus \Omega)$ for some $\eta \in (0,2-2\s)$. Assume that $\Omega$ is a bounded $C^{1,1}$ domain.
	
	Then, there exists an $\alpha_0 \in (0,\s)$, depending only on $n$, $\s$, $\lambda$, $\Lambda$, and $\eta$, such that
	\begin{equation*}
	\norm{w}_{C^{\alpha_0} (\overline{\Omega})} \leq C\bpar{\norm{h}_{L^\infty (\Omega)} + \norm{\varphi}_{C^{2\s+\eta}(\R^n\setminus \Omega)}},
	\end{equation*}
	where $C$ is a constant that depends only on $n$, $\s$, $\lambda$, $\Lambda$, $\eta$, and $\Omega$.
\end{proposition}

Note that this result can be combined with the interior estimate \eqref{Eq:Calpha->Calpha+2sEstimate} to prove that weak solutions are indeed classical solutions.

%%%%%%%%%%%%%%%%%%%%%%%%%%%%%%%%%%%%%%%%%%%%%%%%%%%%%%%%
\subsection{A remark on stability}
\label{Subsec:RemarkStability}
%%%%%%%%%%%%%%%%%%%%%%%%%%%%%%%%%%%%%%%%%%%%%%%%%%%%%%%%

Recall that we say that a bounded solution $w$ to $L_K w = f(w)$ in $\Omega\subset \R^n$ is \emph{stable} in $\Omega$ if the second variation of the energy at $w$ is nonnegative. That is, if
\begin{equation*}
\dfrac{1}{2} \int_{\R^n} \int_{\R^n} |\xi (x) - \xi(y)|^2 K(x - y) \d x \d y - \int_{\Omega} f'(w) \xi^2 \d x \geq 0
\end{equation*}
for every $\xi \in C^\infty_c (\Omega)$.

The following fact regarding stability will be used in Sections~\ref{Sec:SymmetryResults} and \ref{Sec:Asymptotic}. 
Let $w \leq 1$ be a positive solution to $L_K w = f(w)$ in a set $\Omega\subset \R^n$, with $f$ satisfying \eqref{Eq:Hypothesesf}. 
Then $w$ is stable in $\Omega$.
 
The proof of this fact is standard and rather simple, and it is a consequence of the fact that $w$ is a positive supersolution of the linearized operator $L_K - f'(w)$. We present it here for completeness (a more detailed discussion can be found in \cite{HamelRosOtonSireValdinoci}). 
On the one hand, since $f$ is strictly concave in $(0,1)$ and $f(0)=0$, then $f'(w)w<f(w)$ in $\Omega$ (recall that $w$ is positive there). On the other hand, it is easy to check that the following pointwise inequality holds for all functions $\varphi$ and $\xi$, with $\varphi>0$:
\begin{equation}
\label{Eq:IdentityStability}
\big (\varphi(x) - \varphi(y) \big) \bpar{\dfrac{\xi^2(x)}{\varphi(x)} - \dfrac{\xi^2(y)}{\varphi(y)} } \leq |\xi (x) - \xi(y)|^2\,.
\end{equation}
Using these two facts and the symmetry of $K$, for every $\xi\in C^\infty_c(\Omega)$ we have
\begin{align*}
\int_\Omega f'(w) \xi^2 \d x & \leq \int_\Omega  \dfrac{\xi^2}{w} f(w) \d x = \int_\Omega  \dfrac{\xi^2}{w} L_Kw \d x \\ 
&= \dfrac{1}{2} \int_{\R^{2m}} \int_{\R^{2m}} \big ( w(x) - w(y) \big) \bpar{\dfrac{\xi^2(x)}{w(x)} - \dfrac{\xi^2(y)}{w(y)} } K(x - y) \d x \d y
\\ 
&\leq \dfrac{1}{2} \int_{\R^{2m}} \int_{\R^{2m}} |\xi (x) - \xi(y)|^2 K(x - y) \d x \d y\,.
\end{align*}
Thus, $w$ is stable in $\Omega$.

%%%%%%%%%%%%%%%%%%%%%%%%%%%%%%%%%%%%%%%%%%%%%%%%%%%%%%%%
\subsection{Maximum principles for doubly radial odd functions}
\label{Subsec:MaxPrinciples}
%%%%%%%%%%%%%%%%%%%%%%%%%%%%%%%%%%%%%%%%%%%%%%%%%%%%%%%%

In this last subsection, we state the basic maximum principles for doubly radial odd functions. Note that in the following result we only need assumptions on the functions at one side of the Simons cone thanks to their symmetry. This was proved in part~I and follows readily from the expression \eqref{Eq:OperatorOddF} by using the key inequality \eqref{Eq:KernelInequality} for the kernel $\overline{K}$.

\begin{proposition}[Maximum principle for odd functions with respect to $\ccal$ \cite{FelipeSanz-Perela:IntegroDifferentialI}]
	\label{Prop:MaximumPrincipleForOddFunctions} 
	Let $\Omega \subset \ocal$ be an open set and let $L_K$ be an integro-differential operator with a radially symmetric kernel $K$ satisfying the positivity condition \eqref{Eq:KernelInequality} and such that $L_K\in \lcal_0(2m, \s, \lambda, \Lambda)$.  
	Let $w\in C^{\alpha}(\Omega)\cap C(\overline{\Omega})\cap L^\infty(\R^{2m})$, with $\alpha > 2\s$, be a doubly radial function which is odd with respect to the Simons cone. 
	
	\begin{enumerate}[label=(\roman{*})]
		\item  (Weak maximum principle)
		Assume that
		$$
		\beqc{\PDEsystem}
		L_K w + c(x) w & \geq & 0 & \text{ in } \Omega\,,\\
		w & \geq & 0 & \text{ in } \ocal \setminus \Omega\,,
		\eeqc
		$$
		with $c \geq 0$, and that either
		$$
		\Omega \text{ is bounded} \quad \text{ or } \liminf_{x \in \ocal,\,|x|\to +\infty} w(x) \geq 0\,.
		$$
		Then, $w \geq 0$ in $\Omega$.
		
		\item (Strong maximum principle)  
		Assume that $L_K w + c(x) w\geq 0$ in $\Omega$, with $c$ any continuous function, and that $w\geq 0$ in $\ocal$. Then, either $w\equiv 0$ in $\ocal$ or $w > 0$ in $\Omega$.
	\end{enumerate} 
\end{proposition}

\begin{remark}
	\label{Remark:MaxPrincipleSingularity}
	Following the proof of this result in part~I, it is easy to see that the interior regularity assumptions on $w$ in the previous statement can be weakened.
	Indeed, we are assuming that $w\in C^\alpha(\Omega)$ with $\alpha>2\s$ in order to guarantee that $L_K w$ is finite everywhere in $\Omega$.
	Instead of this, we can simply assume that $w$ is Hölder continuous in $\Omega$ (with Hölder exponent arbitrarily small), as long as $L_K w = + \infty$ at the points of $\Omega$ where $w$ is not regular enough for $L_K w$ to be finite. 
	In such case, $L_K w + c(x) w \geq 0$ holds as well and we can proceed with the argument as done in part~I.
	
	Proposition~\ref{Prop:MaximumPrincipleForOddFunctions} with these weaker assumptions on $w$ will used later in the proof of Theorem~\ref{Th:ExistenceUniqueness} (see Remark~\ref{Remark:CsRegularityFirstEigenfunction} below): we will apply it to a function $w$ being no more regular than $C^{\alpha_0}$ at some points in the interior of $\Omega$, where $\alpha_0$ is given by Proposition~\ref{Prop:BoundaryRegularity}. 
\end{remark}

%%%%%%%%%%%%%%%%%%%%%%%%%%%%%%%%%%%%%%%%%%%%%%%%%%%%%%%%%%%%%%%%%%%%%%
%%%%%%%%%%%%%%%%%%%%%%%%%%%%%%%%%%%%%%%%%%%%%%%%%%%%%%%%%%%%%%%%%%%%%%

%%%%%%%%%%%%%%%%%%%%%%%%%%%%%%%%%%%%%%%%%%%%%%%%%%
\section{Existence and uniqueness of the saddle-shaped solution: monotone iteration method}
%%%%%%%%%%%%%%%%%%%%%%%%%%%%%%%%%%%%%%%%%%%%%%%%%%
\label{Sec:ExistenceUniqueness}

In this section we prove the existence and uniqueness result of Theorem~\ref{Th:ExistenceUniqueness}. 
The proof of the existence is based on the maximum principle and the first ingredient that we need is a version of the monotone iteration procedure for doubly radial functions which are odd with respect to the Simons cone $\ccal$.
In order to prove the uniqueness we will use the asymptotic behavior result of Theorem~\ref{Th:AsymptoticBehaviorSaddleSolution} together with the maximum principle for the linearized operator $L_K - f'(u)$, given in Proposition~\ref{Prop:MaximumPrincipleLinearized}; both results will be proved in the subsequent sections.

We next present the monotone iteration method for doubly radial odd functions. 
In this result and along the section, we will call odd sub/supersolutions to problem \eqref{Eq:SemilinearSolutionInBall} the functions that are doubly radial, odd with respect to the Simons cone, and satisfy the corresponding problem in \eqref{Eq:SemilinearSubSuperSolutionInBall}.

\begin{proposition}
	\label{Prop:MonotoneIterationOdd}
	Let $\s\in (0,1)$ and let $K$ be a radially symmetric kernel  satisfying the convexity assumption \eqref{Eq:SqrtConvex} and such that $L_K\in \lcal_0$. Assume that $\vsub \leq \vsup$ are two bounded functions which are doubly radial, odd with respect to the Simons cone, and belonging to $C^{2\s + \varepsilon} (B_R)$ for some $\varepsilon>0$. Furthermore, assume that $\vsub\in C^\varepsilon(\overline{B_R})$ and that $\vsub$ and $\vsup$ satisfy respectively   
	\begin{equation}
	\label{Eq:SemilinearSubSuperSolutionInBall}
	\beqc{\PDEsystem}
	L_K\vsub & \leq & f(\vsub) & \textrm{ in } B_R \cap \ocal\,, \\
	\vsub & \leq & \varphi & \textrm{ in } \ocal \setminus B_R\,, 
	\eeqc
	\quad \textrm{ and } \quad 
	\beqc{\PDEsystem}
	L_K\vsup & \geq & f(\vsup) & \textrm{ in } B_R \cap \ocal\,, \\
	\vsup & \geq & \varphi & \textrm{ in } \ocal \setminus B_R\,, 
	\eeqc
	\end{equation}
	with $f$ a $C^1$ odd function and $\varphi\in C^{2\s+\varepsilon}(\R^n)$ a bounded doubly radial odd function.
	
	Then, there exists a classical solution $v$ to the problem
	\begin{equation}
	\label{Eq:SemilinearSolutionInBall}
	\beqc{\PDEsystem}
	L_K v & = & f(v) & \textrm{ in } B_R\,, \\
	v &=& \varphi &  \textrm{ in } \R^{2m} \setminus B_R\,, 
	\eeqc
	\end{equation}
	such that $v\in C^{2\s+\tilde\varepsilon}(B_R)\cap  C^{\tilde\varepsilon}(\overline{B_R}) $ for some $\tilde\varepsilon>0$, it is doubly radial, odd with respect to the Simons cone, and  $\vsub \leq v \leq \vsup$ in $\ocal$.
\end{proposition}

In the previous statement we required $C^{2\s + \varepsilon}$ regularity on $\vsub$ and $\vsup$ in order to $L_K$ be finite when applied to them.
In view of Remark~\ref{Remark:MaxPrincipleSingularity}, we can relax this assumption, since we do not need the operator to be finite in the whole set $B_R$ when applied to a subsolution (respectively supersolution), it can take the value $-\infty$ (respectively $+\infty$) at some points.
Note, however, that we cannot drop the assumption $\vsub\in C^\varepsilon(\overline{B_R})$ if we want $v$  to have the desired regularity.

\begin{proof}[Proof of Proposition~\ref{Prop:MonotoneIterationOdd}]
	The proof follows the classical monotone iteration method for elliptic equations (see for instance \cite{Evans}). We just give here a sketch of the proof. 
	First, let $M \geq 0$ be such that $-M \leq \vsub \leq \vsup \leq M$ and set
	$$
	b := \max \left \{{0, - \min_{[-M,M]}f'}\right \}\geq 0\,.
	$$
	Then one defines 
	$$
	\widetilde{L}_K w := L_Kw + b w 	\quad \text{ and } \quad 	g(\tau) := f(\tau) + b \tau\,.
	$$
	Therefore, our problem is equivalent to find a solution to
	$$
	\beqc{\PDEsystem}
	\widetilde{L}_Kv & = & g(v) & \textrm{ in } B_R\,, \\
	v &=& \varphi &  \textrm{ in } \R^{2m} \setminus B_R\,, 
	\eeqc
	$$
	such that $v$ is doubly radial, odd with respect to the Simons cone and  $\vsub \leq v \leq \vsup$ in $\ocal$. Here the main point is that $g$ is also odd but satisfies $g'(\tau) \geq 0$ for $\tau \in [-M,M]$. Moreover, since $b \geq 0$, $\widetilde{L}_K$ satisfies the maximum principle for odd functions in $\ocal$ (as in Proposition~\ref{Prop:MaximumPrincipleForOddFunctions}).
	
	We define $v_0 = \vsub$ and, for $k\geq 1$, let $v_k$ be the solution to the linear problem
	$$
	\beqc{\PDEsystem}
	\widetilde{L}_K v_k & = & g(v_{k-1}) & \textrm{ in } B_R\,, \\
	v_k &=& \varphi &  \textrm{ in } \R^{2m} \setminus B_R\,. 
	\eeqc
	$$
	It is easy to see by induction and the regularity results from Proposition~\ref{Prop:InteriorRegularity} that $v_k\in L^\infty(\R^n) \cap C^{2\s+2\tilde\varepsilon}(B_R)\cap C^{2\tilde\varepsilon}(\overline{B_R})$ for some $\tilde\varepsilon>0$. 
	Moreover, given $\Omega\subset B_R$ a compact set, then $\norm{v_k}_{C^{2\s+2\tilde\varepsilon}(\Omega)}$ is uniformly bounded in $k$.
	
	Then, using the maximum principle it is not difficult to show by induction that 
	$$
	\vsub = v_0 \leq v_1 \leq \ldots \leq v_k \leq v_{k+1} \leq \ldots \vsup \quad \text{ in }\ocal\,,
	$$
	and that each function $v
	_k$ is doubly radial and odd with respect to $\ccal$. Finally, by the Arzelà-Ascoli theorem and the compact embedding of H\"older spaces we see that, up to a subsequence, $v_k$ converges to the desired solution $v\in C^{2\s+\tilde\varepsilon}(B_R)\cap  C^{\tilde\varepsilon}(\overline{B_R}) $.
\end{proof}

In order to construct a positive subsolution to \eqref{Eq:SemilinearSolutionInBall} with zero exterior data, we also need a characterization and some properties of the first odd eigenfunction and eigenvalue for the operator $L_K$, which are presented next. This eigenfunction is obtained though a minimization of the Rayleigh quotient in the appropriate space, defined next.

Given a set $\Omega \subset \R^{2m}$ and a translation invariant and positive kernel $K$, we define the space
$$
\H^K_0(\Omega) := \setcond{w \in L^2(\Omega)}{w = 0 \quad \textrm{a.e. in } \R^{2m} \setminus \Omega \quad \textrm{ and } [w]^2_{\H^K(\R^{2m})} < + \infty},
$$
where
\begin{equation}
\label{Eq:SeminormHK}
[w]^2_{\H^K(\R^{2m})} := \dfrac{1}{2}\int_{\R^{2m}} \int_{\R^{2m}} |w(x) - w(y)|^2 K(x-y) \d x \d y\,.
\end{equation}
Recall also that when $K$ satisfies the ellipticity assumption \eqref{Eq:Ellipticity}, then $\H^K_0 (\Omega) = \H^\s_0 (\Omega)$, which is the space associated to the kernel of the fractional Laplacian, $K(y) = c_{n,\s}|y|^{-n-2\s}$. We also define, for $\Omega$ doubly radial and symmetric with respect to $\ccal$, the space
$$
\widetilde{\H}^K_{0, \, \mathrm{odd}}(\Omega) := \setcond{w \in \H^K_0(\Omega)}{w \textrm{ is doubly radial a.e. and odd with respect to } \ccal}.
$$
Recall that when $K$ is radially symmetric and $w$ is doubly radial, we can replace the kernel $K(x-y)$ in the definition \eqref{Eq:SeminormHK} by the kernel $\overline{K}(x,y)$. This is readily deduced after a change of variables and taking the mean among all $R\in O(m)^2$ (see the details in Section~3 of \cite{FelipeSanz-Perela:IntegroDifferentialI}).

\begin{lemma}
	\label{Lemma:FirstOddEigenfunction}
	Let $\Omega\subset \R^{2m} $ be a bounded set of double revolution and let  $K$ be a radially symmetric kernel satisfying the positivity condition \eqref{Eq:KernelInequality} and such that $L_K\in \lcal_0(2m, \s, \lambda, \Lambda)$. Let us define 
	$$
	\lambda_{1, \, \mathrm{odd}}(\Omega, L_K) := \inf_{w \in \widetilde{\H}^K_{0, \, \mathrm{odd}}(\Omega)} \dfrac{\dfrac{1}{2}  \ds\int_{\R^{2m}} \int_{\R^{2m}} |w(x) - w(y)|^2 \overline{K}(x,y) \d x \d y}{ \ds \int_\Omega w(x)^2 \d x}\,.
	$$
	
	Then, such infimum is attained at a function $\phi_1\in \widetilde{\H}^K_{0, \, \mathrm{odd}}(\Omega)\cap L^\infty(\Omega)$ which solves
	$$
	\beqc{\PDEsystem}
	L_K \phi_1 &=& \lambda_{1, \, \mathrm{odd}}(\Omega, L_K) \phi_1 & \textrm{ in } \Omega\,,\\
	\phi_1 & = & 0 & \textrm{ in } \R^{2m}\setminus \Omega\,,
	\eeqc
	$$
	and satisfies that $\phi_1 > 0$ in $\Omega \cap \ocal$.	We call this function $\phi_1$ the \emph{first odd eigenfunction of $L_K$ in $\Omega$}, and $\lambda_{1, \, \mathrm{odd}}(\Omega, L_K) $, the \emph{first odd eigenvalue}. 
	
	Moreover, in the case $\Omega = B_R$, there exists a constant $C$ depending only on $n$, $\s$, and $\Lambda$, such that
	$$
	\lambda_{1, \, \mathrm{odd}}(B_R, L_K) \leq C R^{-2\s}\,. 
	$$ 
\end{lemma}

\begin{proof}
	The first two statements are deduced exactly as in Proposition~9 of \cite{ServadeiValdinoci}, using the same arguments as in  Lemma~3.4 of \cite{FelipeSanz-Perela:IntegroDifferentialI} to guarantee that $\phi_1$ is nonnegative in $\ocal$. The fact that $\phi_1 > 0$ in $\Omega \cap \ocal$ follows from the strong maximum principle (see Proposition~\ref{Prop:MaximumPrincipleForOddFunctions}).
	
	We show the third statement. Let $\widetilde{w} (x):= w(Rx)$ for every $w\in \widetilde{\H}^K_{0, \, \mathrm{odd}}(B_R)$. Then,
	\begin{align*}
	& \min_{w \in \widetilde{\H}^K_{0, \, \mathrm{odd}}(B_R)} \dfrac{\dfrac{1}{2}  \ds\int_{\R^{2m}} \int_{\R^{2m}} |w(x) - w(y)|^2 \overline{K}(x,y) \d x \d y}{ \ds \int_{B_R} w(x)^2 \d x} \quad \quad \quad \quad \quad \quad \quad \quad \quad \quad \quad \quad\\
	&   \quad \quad \quad \quad \quad \quad \leq \min_{\widetilde{w} \in \widetilde{\H}^K_{0, \, \mathrm{odd}}(B_1)} \dfrac{\dfrac{c_{n, \s}\Lambda}{2}  \ds\int_{\R^{2m}} \int_{\R^{2m}} |\widetilde{w}(x/R) - \widetilde{w}(y/R)|^2 |x - y|^{-n-2 \s}\d x \d y}{ \ds \int_{B_R} \widetilde{w}(x/R)^2 \d x}
	\\
	& \quad \quad \quad \quad \quad \quad = R^{-2 \s }\min_{\widetilde{w} \in \widetilde{\H}^s_{0, \, \mathrm{odd}}(B_1)} \dfrac{\dfrac{c_{n, \s}\Lambda}{2}  \ds\int_{\R^{2m}} \int_{\R^{2m}} |\widetilde{w}(x) - \widetilde{w}(y)|^2 |x - y|^{-n-2 \s}\d x \d y}{ \ds \int_{B_1} \widetilde{w}(x)^2 \d x}
	\\
	& \quad \quad \quad \quad \quad \quad = \lambda_{1, \, \mathrm{odd}}(B_1, \fraclaplacian) \Lambda R^{-2 \s } \,.
	\end{align*}
\end{proof}

\begin{remark}
	\label{Remark:CsRegularityFirstEigenfunction}
	Note that, by the regularity results for $L_K$ stated in Section~\ref{Sec:Preliminaries}, we have that $\phi_1 \in C^{\alpha_0}(\overline{\Omega})\cap C^{\alpha_0 + 2\s}(\Omega)$ for some $0<\alpha_0<\s$, and the regularity up to the boundary is optimal. Due to this and the fact that $\phi_1 >0$ in $\Omega\cap \ocal$ while $\phi_1=0$ in $\R^{2m}\setminus \Omega$, it is easy to check by using \eqref{Eq:OperatorOddF} that $-\infty <L_K \phi_1 < 0$ in $\ocal\setminus \overline{\Omega}$ and that $L_K \phi_1 = -\infty$ on $\partial \Omega \cap \ocal$.
\end{remark}

With these ingredients, we can proceed with the proof of  Theorem~\ref{Th:ExistenceUniqueness}.

\begin{proof}[Proof of Theorem~\ref{Th:ExistenceUniqueness}] We divide it into two parts.

\textbf{\textit{i}) Existence:}
	The strategy is to build a suitable solution $u_R$ of 
	\begin{equation}
	\label{Eq:ProofExistenceProblemBR}
	\beqc{\PDEsystem}
	L_K u_R &=& f(u_R) & \textrm{ in } B_R\,,\\
	u_R &=& 0 & \textrm{ in }\R^{2m} \setminus B_R\,,
	\eeqc
	\end{equation}
	and then let $R\to+ \infty$ to get a saddle-shaped solution.
	
	Let $\phi_1^{R_0}$ be the first odd eigenfunction of $L_K$ in $B_{R_0} \subset \R^{2m}$, given by Lemma~\ref{Lemma:FirstOddEigenfunction}, and let  $\lambda_1^{R_0} := \lambda_{1, \, \mathrm{odd}}(B_{R_0}, L_K)$. We claim that for $R_0$ big enough and $\varepsilon>0$ small enough, $\usub_R := \varepsilon\phi_1^{R_0} $ is an odd subsolution of \eqref{Eq:ProofExistenceProblemBR} for every $R\geq R_0$. To see this, note first that, without loss of generality, we can assume that $\norm{\phi_1^{R_0}}_{L^\infty(B_R)}=1$. Now, since $f$ is strictly concave in $(0,1)$ and $f(0)=0$, we have that $f'(\tau)\tau<f(\tau)$ for all $\tau>0$. Thus, using that $\varepsilon \phi_1^{R_0}>0$ in $B_{R_0}\cap \ocal$, it follows that for every $x\in B_{R_0}\cap \ocal$,
	$$
	\dfrac{f(\varepsilon \phi_1^{R_0}(x))}{\varepsilon \phi_1^{R_0}(x)} > f'(\varepsilon \phi_1^{R_0}(x)) \geq f'(0)/2
	$$
	if $\varepsilon$ is small enough, independently of $x$ (recall that we assumed $|\phi_1|\leq 1$). Therefore, since $f'(0)>0$, taking $R_0$ big enough so that $\lambda_1^{R_0} < f'(0)/2$ (this can be achieved thanks to the last statement of Lemma~\ref{Lemma:FirstOddEigenfunction}), we have that for every $x\in B_{R_0}\cap \ocal$,  $f(\varepsilon \phi_1^{R_0}(x)) > \lambda_1^{R_0}  \varepsilon \phi_1^{R_0}(x)$. Thus,
	$$
	L_K \usub_R = \lambda_1^{R_0} \varepsilon \phi_1^{R_0} < f(\varepsilon\phi_1^{R_0}) = f(\usub_R) \quad \textrm{ in } B_{R_0}\cap \ocal\,.
	$$
	In addition, if $x\in (B_R\setminus B_{R_0})\cap\ocal$, by Remark~\ref{Remark:CsRegularityFirstEigenfunction} we have that
	$$
	L_K \usub_R < 0 = f(0) =  f(\usub_R) \quad \textrm{ in } (B_R\setminus B_{R_0})\cap \ocal\,.
	$$
	Note that in $\partial B_{R_0}$ we have $L_K \usub_R = -\infty$. Hence, the claim is proved.
	
	Now, if we define $\usup_R := \chi_{\ocal \cap B_R} - \chi_{\ical \cap B_R}$, a simple computation shows that it is an odd supersolution to \eqref{Eq:ProofExistenceProblemBR}. Therefore, using the monotone iteration procedure given in Proposition~\ref{Prop:MonotoneIterationOdd} (taking into account Remarks~\ref{Remark:MaxPrincipleSingularity} and \ref{Remark:CsRegularityFirstEigenfunction} when using the maximum principle), we obtain a solution $u_R$ to \eqref{Eq:ProofExistenceProblemBR} such that it is doubly radial, odd with respect to the Simons cone, and $\varepsilon \phi_1^{R_0} = \usub_R \leq u_R \leq \usup_R$ in $\ocal$. Note that, since $\usub_R > 0$ in $\ocal \cap B_{R_0}$, the same holds for $u_R$.
	
	Using a standard compactness argument, we let $R\to +\infty$ to obtain a sequence $u_{R_j}$ converging on compacts in  $C^{2\s + \eta}(\R^{2m})$ norm, for some $\eta > 0$, to a solution $u \in C^{2\s + \eta}(\R^{2m})$ of $L_K u = f(u)$ in $\R^{2m}$. Note that $u$ is doubly radial, odd with respect to the Simons cone and $0\leq u \leq 1$ in $\ocal$. 
	Let us show that $0 < u < 1$ in $\ocal$, which will yield that $u$ is a saddle-shaped solution. 
	By the usual strong maximum principle it follows readily that $u<1$ in $\ocal$. 
	Moreover, since $u_R\geq\varepsilon \phi_1^{R_0}>0$ in  $\ocal \cap B_{R_0}$ for $R>R_0$, this holds also the limit, that is, $u\geq\varepsilon \phi_1^{R_0}>0$ in  $\ocal \cap B_{R_0}$. 
	Therefore, by applying the strong maximum principle for odd functions (see Proposition~\ref{Prop:MaximumPrincipleForOddFunctions}) we obtain that $0 < u < 1$ in $\ocal$.
	
	\medskip
	
	\textbf{\textit{ii}) Uniqueness:}
	Let $u_1$ and $u_2$ be two saddle-shaped solutions. Define $v := u_1 - u_2$, which is a doubly radial function that is odd with respect to $\ccal$. Then,
	$$
	L_K v = f(u_1) - f(u_2) \leq f'(u_2) (u_1 - u_2) = f'(u_2) v \quad \textrm{ in } \ocal\,,
	$$
	since $f$ is concave in $(0,1)$. Moreover, by the asymptotic result (see Theorem~\ref{Th:AsymptoticBehaviorSaddleSolution}), we have
	$$
	\limsup_{x\in \ocal, \ |x|\to \infty} v(x) = 0\,.
	$$
	Then, by the maximum principle in $\ocal$ for the linearized operator $L_K  - f'(u_2)$ (see Proposition~\ref{Prop:MaximumPrincipleLinearized}), it follows that $v \leq 0$ in $\ocal$, which means $u_1 \leq u_2$ in $\ocal$. Repeating the  argument with $-v = u_2 - u_1$ we deduce $u_1 \geq u_2$ in $\ocal$. Therefore, $u_1 = u_2$ in $\R^{2m}$.	
\end{proof}

\begin{remark}
	\label{Remark:uStableinO}
	Since the saddle-shaped solution $u$ is positive in $\ocal$, it follows that $u$ is stable in this set, as explained in Section~\ref{Sec:Preliminaries}. 
	This fact will be used in Section~\ref{Sec:Asymptotic}.
\end{remark}

%%%%%%%%%%%%%%%%%%%%%%%%%%%%%
\section{Symmetry and Liouville type results}
\label{Sec:SymmetryResults}
%%%%%%%%%%%%%%%%%%%%%%%%%%%%%

This section is devoted to prove the Liouville type result of Theorem~\ref{Th:LiouvilleSemilinearWholeSpace} and the one-dimensional symmetry result of Theorem~\ref{Th:SymmHalfSpace}. Both of them will be needed in the following section to establish the asymptotic behavior of the saddle-shaped solution. 

\subsection{A Liouville type result for positive solutions in the whole space}

In the proof of Theorem~\ref{Th:LiouvilleSemilinearWholeSpace} we will need two main ingredients, that we present next. The first one is a Harnack inequality for solutions to the semilinear problem \eqref{Eq:PositiveWholeSpace}. This inequality follows readily from the results of Cozzi in \cite{Cozzi-DeGiorgiClassesLong}, although the precise result that we need is not stated there. For the reader's convenience and for future reference, we present the result here and indicate how to deduce it from the results in \cite{Cozzi-DeGiorgiClassesLong}.

\begin{proposition}
	\label{Prop:HarnackSemilinear}
	Let $L_K\in\mathcal{L}_0(n,\s,\lambda,\Lambda)$ and let $w$ be a solution to \eqref{Eq:PositiveWholeSpace} with $f$ a Lipschitz nonlinearity such that $f(0) = 0$. Then, for every $x_0 \in \R^n$ and every $R>0$, it holds
	$$
	\sup_{B_R(x_0)} w \leq C  \inf_{B_R(x_0)} w, 
	$$
	with $C>0$  depending only on $n,\s,\lambda,\Lambda$, and $R$.
\end{proposition}

\begin{proof}
	Following the notation of \cite{Cozzi-DeGiorgiClassesLong}, since $f$ is Lipschitz and $f(0) = 0$, we have
	$$
	|f(u)|\leq d_1 + d_2 |u|^{q-1} \quad \text{ in } \R^n\,,
	$$ 
	with $d_1=0$, $d_2 =\norm{f}_{\mathrm{Lip}}$ and $q=2$. With this choice of the parameters, we only need to repeat the proof of Proposition~8.5 in \cite{Cozzi-DeGiorgiClassesLong} (with $p=2$ and $\Omega = \R^n$) in order to obtain that $u$ belongs to the fractional De Giorgi class $\mathrm{DG}^{\s,2} ( \R^n , 0, H, -\infty,2\s/n,2\s,+\infty)$	for some constant $H>0$ (see \cite{Cozzi-DeGiorgiClassesLong} for the precise definition of these classes). Therefore, the Harnack inequality follows from Theorem~6.9 in \cite{Cozzi-DeGiorgiClassesLong}.
\end{proof}

The second ingredient that we need in the proof of Theorem~\ref{Th:LiouvilleSemilinearWholeSpace} is the following parabolic maximum principle in the unbounded set $\R^n \times (0,+\infty)$. 

\begin{proposition}
	\label{Prop:ParaMaxPrp}
	Let $L_K \in \lcal_0(n,\s,\lambda, \Lambda)$ and let $v$ be a bounded function, $C^\alpha$ with $\alpha > 2\s$ in space and $C^1$ in time, such that
	\begin{equation*}
	\beqc{\PDEsystem}
	\partial_t v + L_K  v + c(x)\,v &\leq& 0 & \textrm{ in }\R^n\times(0,+\infty)\,,\\
	v(x,0) &\leq& 0 & \textrm{ in } \R^n\,,
	\eeqc
	\end{equation*}
	with $c(x)$ a continuous and bounded function. Then,
	$$ v(x,t) \leq 0 \,\,\,\,\,\text{ in } \,\,\, \R^n\times[0,+\infty). $$
\end{proposition}

This result can be deduced from the usual parabolic maximum principle in a bounded (in space and time) set with a rather simple argument. Since we have not found a specific reference where such result is stated, let us present its proof with full detail for the sake of clarity. First of all, we present the usual parabolic maximum principle in a bounded set in $\R^n \times (0,+\infty)$. The proof for cylindrical sets $\Omega \times (0,T)$ can be found for instance in \cite{BarriosPeralSoriaValdinoci}. Although the argument for general bounded sets is essentially the same,  we include here a short proof for the sake of completeness.

\begin{lemma}
	\label{Lemma:ParabolicmaxPrpBdd}
	Let $\Omega  \subset B_R\times(0,T) \subset \R^n \times (0,+\infty)$ be a bounded open set.
	Let $L_K$ be an integro-differential operator of the form \eqref{Eq:DefOfLu} with a symmetric kernel satisfying \eqref{Eq:Ellipticity}, and let $v$ be a bounded function, $C^\alpha$ with $\alpha > 2\s$ in space and $C^1$ in time, satisfying
	\begin{equation*}
	\beqc{\PDEsystem}
	\partial_t v + L_K v &\leq& 0 & \textrm{ in } \Omega \subset B_R\times(0,T)\,,\\
	v(x,0) &\leq& 0 & \textrm{ in } \overline{\Omega} \cap \{t=0\} \subset B_R\,,\\
	v &\leq& 0 & \textrm{ in } ( \R^n \times (0,T))\setminus \Omega \,.
	\eeqc
	\end{equation*}
	Then, $v\leq 0$ in $\R^n\times [0,T]$.
\end{lemma}

\begin{proof}
	By contradiction, for every small $\varepsilon > 0$ assume that 
	$$
	M:=\sup_{\R^n \times (0,T-\varepsilon)}v > 0.
	$$
	By the sign of the initial condition and since $v \leq 0 $ in $(\R^n \times (0,T))\setminus \Omega$, $v$  attains this positive value $M$ at a point $(x_0,t_0) \in \Omega$ with $t_0\leq T-\varepsilon$. If $t_0\in(0,T-\varepsilon)$, then $(x_0,t_0)$ is an interior global maximum (in $\R^n \times (0,T-\varepsilon)$) and it must satisfy $v_t(x_0,t_0)=0$ and $L_K v(x_0,t_0)>0$, which contradicts the equation. If $t_0 = T-\varepsilon$, then $v_t(x_0,t_0)\geq 0$ and $L_K v(x_0,t_0)>0$, which is also a contradiction with the equation. Thus, $v\leq 0$ in $\R^n\times [0,T-\varepsilon)$ and since this holds for $\varepsilon>0$ arbitrarily small, we deduce $v\leq 0$ in $\R^n\times [0,T)$, and by continuity, in $\R^n\times [0,T]$.
\end{proof}

To establish Proposition~\ref{Prop:ParaMaxPrp} from Lemma~\ref{Lemma:ParabolicmaxPrpBdd}, we need to introduce an auxiliary function enjoying certain properties (see Lemma~\ref{Lemma:SolBallToZero} below). Before presenting it, we need the following result.

\begin{lemma}
	\label{Lemma:NoBddSolL=1}
	There is no bounded solution to $L_K v=1$ in $\R^n$ for any $L_K \in \lcal_0$.
\end{lemma}

\begin{proof}
	Assume by contradiction that such solution exists. Then, by interior regularity (see Section~\ref{Sec:Preliminaries}) $v\in C^1(\R^n)$ and $|\nabla v|\leq C$ in $\R^n$. For every $i = 1,\ldots, n$, we differentiate the equation with respect to $x_i$ to obtain
	\begin{equation*}
	\beqc{\PDEsystem}
	L_K  v_{x_i} &=& 0 & \textrm{ in } \R^n\,,\\
	|v_{x_i}| &\leq& C & \textrm{ in } \R^n\,.
	\eeqc
	\end{equation*}
	By the Liouville theorem for the operator $L_K $ (it is proved exactly as in \cite{RosOtonSerra-Stable}, see also \cite{SerraC2s+alphaRegularity}), $v_{x_i}$ is constant. Hence, $\nabla v$ is constant, and thus $v$ is affine. But since $v$ is bounded, $v$ must be constant, and we arrive at a contradiction with $L_K v=1$.
\end{proof}

With this result we can introduce the auxiliary function that we will use to prove the parabolic maximum principle of Proposition~\ref{Prop:ParaMaxPrp}.

\begin{lemma}
	\label{Lemma:SolBallToZero}
	Let $L_K \in \lcal_0(n,\s,\lambda, \Lambda)$. Then, for every $R>0$ there exists a constant $M_R>0$ and a continuous function $\psi_R\geq 0$ solution to
	\begin{equation}
	\label{Eq:psiRProblem}
	\beqc{\PDEsystem}
	L_K  \psi_R &=& -1/M_R & \textrm{ in } B_R\,,\\
	\psi_R &=& 1 & \textrm{ in } \R^n\setminus B_R\,,
	\eeqc
	\end{equation}
	satisfying 
	$$
	\psi_R \to  0 \ \text{ pointwise }  \text{ and } \quad M_R  \to +\infty \ \quad \text{ as } R\to +\infty\,.
	$$
\end{lemma}

\begin{proof}
	First, consider $\phi_R$ the solution to
	\begin{equation*}
	\beqc{\PDEsystem}
	L_K  \phi_R &=& 1 & \textrm{ in } B_R\,,\\
	\phi_R &=& 0 & \textrm{ in } \R^n\setminus B_R\,.
	\eeqc
	\end{equation*}
	Note that the existence of a weak solution to the previous problem is given by the Riesz representation theorem. Moreover, by standard regularity results (see Section \ref{Subsec:Regularity}), $\phi_R$ is in fact a classical solution and by the maximum principle, $\phi_R>0$ in $B_R$.
	
	Define $M_R := \sup_{B_R} \phi_R$. Since $M_R$ is increasing (to check this, use the maximum principle to compare $\phi_R$ and $\phi_{R'}$ with $R>R'$), it must have a limit $M\in \R \cup \{+\infty\}$. Assume by contradiction that $M<+\infty$ and consider the new function $ \varphi_R := \phi_R/M_R$, which satisfies
	\begin{equation}
	\label{Eq:varphiRProblem}
	\beqc{\PDEsystem}
	L_K  \varphi_R &=& 1/M_R & \textrm{ in } B_R\,,\\
	\varphi_R &=& 0 & \textrm{ in } \R^n\setminus B_R\,, \\
	\varphi_R &\leq & 1\,.
	\eeqc
	\end{equation}
	By a standard compactness argument, we deduce that as $R\to +\infty$, $\varphi_R$ converges (up to a subsequence) to a function $\varphi$ that solves $L_K  \varphi = 1/M$ in $\R^n$ and satisfies  $|\varphi| \leq 1$. This contradicts Lemma~\ref{Lemma:NoBddSolL=1} and therefore, $M_R \to +\infty$ as $R\to +\infty$. 
	
	Define now $\psi_R := 1-\phi_R/M_R = 1-\varphi_R$, which solves trivially \eqref{Eq:psiRProblem}. Thus, it only remains to show that $\psi_R \to 0$ as $R\to +\infty$. We will see that $\varphi_R \to	1$  as $R\to +\infty$. Recall that $\varphi_R$ solves problem \eqref{Eq:varphiRProblem}, and by the previous arguments, by letting $R\to +\infty$ we have that a subsequence of $\varphi_R$ converges uniformly in compact sets to a bounded function $\varphi\geq 0$ that solves $ L_K \varphi = 0 $ in $\R^n$. By the Liouville theorem, $\varphi$ must be constant, and since its $L^\infty$ norm is $1$ and $\varphi\geq 0$, we conclude $\varphi\equiv 1$.	
\end{proof}

With these ingredients, we establish now the parabolic maximum principle in $\R^n \times (0,+\infty)$. 

\begin{proof}[Proof of Proposition~\ref{Prop:ParaMaxPrp}]
	First of all, note that with the change of function $\tilde{v}(x,t) = \e^{-\alpha\,t} v(x,t)$ we can reduce the initial problem in the statement of Proposition~\ref{Prop:ParaMaxPrp} to
	\begin{equation*}
	\beqc{\PDEsystem}
	\partial_t \tilde{v} + L_K  \tilde{v} &\leq& 0 & \textrm{ in } \Omega \subset\R^n\times(0,+\infty)\,,\\
	\tilde{v} &\leq& 0 & \textrm{ in }  \left(\R^n\times(0,+\infty)\right) \setminus  \Omega\,,\\
	\tilde{v}(x,0)&\leq& 0 & \textrm{ in } \R^n\,,
	\eeqc
	\end{equation*}
	if we take $\alpha > \norm{c}_{L^\infty}$ and $\Omega := \{(x,t)\in \R^n\times(0,+\infty) \ : \ v(x,t) > 0\}$.
	
	Now, consider the function 
	$$ 
	w_R(x,t) := \norm{\tilde{v} }_{L^\infty(\R^n \times (0,+\infty))} \left( \psi_R + \dfrac{t}{M_R} \right)\,,
	$$
	where $\psi_R$ and $M_R$ are defined in Lemma~\ref{Lemma:SolBallToZero}. Then, it is easy to check that $w_R$ satisfies
	\begin{equation*}
	\beqc{\PDEsystem}
	\partial_t w_R + L_K  w_R &=& 0 & \textrm{ in }B_R\times(0,T)\,,\\
	w_R(x,0) &\geq& 0 & \textrm{ in } B_R\,,\\
	w_R(x,t) &\geq& \norm{\tilde{v}}_{L^\infty(\R^n \times (0,+\infty))}  & \textrm{ in } \left( \R^n\setminus B_R\right) \times (0,T) \,,
	\eeqc
	\end{equation*}
	for every $T>0$ and $R>0$. Since $w_R \geq 0 \geq \tilde{v}$ in  $\left(\R^n\times(0,+\infty)\right)\setminus \Omega$, by the maximum principle in $(B_R\times (0,T))\cap \Omega$ (see Lemma \ref{Lemma:ParabolicmaxPrpBdd}) we can easily deduce that $ w_R\geq \tilde{v} $ in $B_R\times(0,T)$.
	
	Finally, given an arbitrary point $(x_0,t_0) \in \Omega$, take $R_0>0$ and $T>0$ such that $(x_0,t_0)\in B_{R_0}\times (0,T)$. Thus,
	$$ 
	\tilde{v}(x_0,t_0) \leq w_R(x_0,t_0) =  \norm{\tilde{v} }_{L^\infty(\R^n \times (0,+\infty))} \left( \psi_R (x_0) + \dfrac{t_0}{M_R} \right), \,\,\,\,\,\text{ for every }\,\,\, R\geq R_0.
	$$
	Letting $R \to +\infty$ and using that $\psi_R(x_0) \to 0$ and $M_R \to +\infty$ (see Lemma~\ref{Lemma:SolBallToZero}), we conclude $ \tilde{v}(x_0,t_0) \leq  0$, and therefore $ v(x_0,t_0) = \e^{\alpha\,t_0}\,\tilde{v}(x_0,t_0) \leq 0$.
\end{proof}

%%%%%%%%%%%%%%%%%%%%%%%%%%%%%%%%%%%%%%%%%%%%%%%%%%%%%%%%%%%%%%%%%%%%%%%%%%%%%%%%%%%%%%%%%%%%%%%%%%%%%%%%%%%%%%%%%%%%%%%%%%%%%%%%%%%%%%%%%%%%%%%%%%%

By using the Harnack inequality and the parabolic maximum principle we can now establish Theorem~\ref{Th:LiouvilleSemilinearWholeSpace}. The proof follows the ideas of Berestycki, Hamel, and Nadirashvili from Theorem 2.2 in \cite{BerestyckiHamelNadi} but adapted to the whole space and with an integro-differential operator.

\begin{proof}[Proof of Theorem~\ref{Th:LiouvilleSemilinearWholeSpace}]

	Assume $v\not\equiv 0$. Then, by the strong maximum principle $v>0$. Our goal is to show that $v\equiv 1$, and this will be accomplished in two steps.
	
	\textbf{Step 1: We show that $m:=\inf_{\R^n} v >0$.} 
	
	By contradiction, we assume $m=0$. Then, there exists a sequence $\{x_k\}_{k\in\N}$ such that $v(x_k)\rightarrow 0$ as $k \rightarrow +\infty$.
	
	On the one hand, by the Harnack inequality of Proposition~\ref{Prop:HarnackSemilinear}, given any $R>0$ we have 
	\begin{equation}
	\label{Eq:Harnack}
	\sup_{B_R(x_k)}v \leq C_R \inf_{B_R(x_k)}v \leq C_R \, v(x_k) \rightarrow 0 \,\,\text{as}\,\, k\rightarrow +\infty.
	\end{equation}
	Moreover, since $f(0) = 0 $ and $f'(0)>0$, it is easy to show that $f(t)\geq f'(0)t/2$ if $t$ is small enough. Therefore, from this and \eqref{Eq:Harnack}  we deduce that there exists $M(R)\in\N$ such that
	\begin{align}
	\label{Eq:WholeSpace2}
	L_K  v - \frac{f'(0)}{2}v \geq 0 \,\,\textrm{ in }\ B_R(x_{M(R)})\,.
	\end{align}

	On the other hand, let us define
	$$  \lambda_R^{x_0} = \inf_{\substack{\varphi\in C^1_c(B_R(x_0))\\ \varphi\not\equiv 0}} \frac{\ds  \dfrac{1}{2}\int_{\R^n}\int_{\R^n}|\varphi(x)-\varphi(y)|^2\,K(x-y) \d x \d y}{\ds \int_{\R^n}\varphi(x)^2 \d x}, 
	$$
	which decreases to zero uniformly in $x_0$ as $R\to +\infty$ from being $L_K \in\mathcal{L}_0$ (see the proof of Lemma~\ref{Lemma:FirstOddEigenfunction} and also Proposition~9 of \cite{ServadeiValdinoci}). Therefore, there exists $R_0>0$ such that $ \lambda_R^x < f'(0)/2$ for all $x\in \R^n$ and $R\geq R_0$. In particular, by choosing $x=x_{M(R_0)}$ there exists $w\in C^1_c(B_{R_0}(x_{M(R_0)}))$ such that $w\not\equiv 0$ and
	\begin{equation}
	\label{Eq:Eigenfunction}
	\dfrac{1}{2}\int_{\R^n}\int_{\R^n}|w(x)-w(y)|^2\,K(x-y) \d x \d y < \frac{f'(0)}{2}\int_{\R^n}w^2 \d x.
	\end{equation}
	
	Finally, to get the contradiction, multiply \eqref{Eq:WholeSpace2} by $w^2/v\geq 0$ and integrate in $\R^n$. After symmetrizing the integral involving $L_K$ we get
	\begin{align*}
	0 &\leq \int_{\R^n} \frac{w^2}{v}  L_K v \d x - \frac{f'(0)}{2}\int_{\R^n} w^2 \d x \\
	&= \dfrac{1}{2}\int_{\R^n}\int_{\R^n}\big( v(x)-v(y) \big) \left( \frac{w^2(x)}{v(x)}-\frac{w^2(y)}{v(y)} \right) K(x-y) \d x \d y - \frac{f'(0)}{2}\int_{\R^n} w^2 \d x \\
	&\leq \dfrac{1}{2} \int_{\R^n}\int_{\R^n} |w(x)-w(y)|^2 K(x-y) \d x \d y - \frac{f'(0)}{2}\int_{\R^n} w^2 \d x ,
	\end{align*}
	which contradicts \eqref{Eq:Eigenfunction}. Here we have used that the kernel is positive and symmetric and the inequality \eqref{Eq:IdentityStability}. Therefore, $\inf_{\R^n} v >0$.
	
	\textbf{Step 2: We show that $v\equiv 1$.}
	
	Choose $0<\xi_0<\min\{1,m\}$, which is well defined by Step~1, and let $\xi(t)$ be the solution of the ODE
	$$
	\beqc{\PDEsystem}
	\dot{\xi}(t) &=& f(\xi(t)) & \textrm{ in }(0,+\infty)\,,\\
	\xi(0) &=& \xi_0\,.
	\eeqc
	$$
	Since $f>0$ in $(0,1)$ and $f(1) = 0$ we have that $\dot{\xi}(t)>0$ for all $t\geq 0$, and $\ds \lim_{t\to +\infty} \xi(t) = 1$.
	
	Now, note that both $v(x)$ and $\xi(t)$ solve the parabolic equation
	$$ \partial_t w + L_K w = f(w) \,\,\, \textrm{ in }\R^n\times (0,+\infty)\,, $$
	and satisfy
	$$ v(x) \geq m \geq \xi_0 = \xi(0). $$
	Thus, by the parabolic maximum principle (Proposition~\ref{Prop:ParaMaxPrp}) applied to $v-\xi$, taking $c(x) = -\big(f(v)-f(\xi)\big)/(v-\xi)$, we deduce that $v(x)\geq \xi(t)$ for all $x\in\R^n$ and $t\in(0,\infty)$. By letting $t \to +\infty$ we obtain
	$$
	v(x) \geq 1 \,\, \textrm{ in }\R^n\,.  
	$$
	In a similar way, taking $\tilde{\xi}_0>\norm{v}_{L^\infty} \geq 1$, using $f<0$ in $(1,+\infty)$, $f(1)=0$ and the parabolic maximum principle, we obtain the upper bound $v\leq 1$.
\end{proof}

%%%%%%%%%%%%%%%%%%%%%%%%%%%%%%%%%%%%%%%%%%%%%%%%%%%%%%%%%%%%%%%%%%%%%%%%%%%%%%%%%%%%%%%%%%%%%%%%%%%%%%%%%%%%%%%%%%%%%%%%%%%%%%%%%%%%%%%%%%%%%%%%%%%%%%%%%

\subsection{A one-dimensional symmetry result for positive solutions in a half-space}

In this subsection we establish Theorem~\ref{Th:SymmHalfSpace}. To do it, we proceed in three steps. First, we show that the solution is monotone in the $x_n$ direction by using a moving planes argument (see Proposition~\ref{Prop:MonotonyHalfSpace} below). Once this is shown, we can deduce that the solution $v$ has uniform limits as $x_n\pm\to \infty$. Finally, by using the sliding method (see Proposition~\ref{Prop:HalfSpaceLimUnif} below), we deduce the one-dimensional symmetry of the solution.

We proceed now with the details of the arguments. As we have said, the first step is to show that the solution is monotone. We establish the following result.

\begin{proposition}
	\label{Prop:MonotonyHalfSpace}
	Let $v$ be a bounded solution to one of the problems \eqref{Eq:P1} or \eqref{Eq:P2}, with $L_K \in \lcal_0$ such that the kernel $K$ is nonincreasing in the direction of $x_n$ in $\R^n_+$, that is, 
	$$
	K(x_H-y_H,x_n-y_n) \geq K(x_H-y_H,x_n+y_n) \,\,\,\,\text{for all } \,\, x,y\in \R^n_+.
	$$
	Let $f$ be a Lipschitz nonlinearity such that $f>0$ in $(0,\norm{v}_{L^\infty(\R^n_+)})$. 
	
	Then,
	$$
	\frac{\partial v}{\partial x_n} > 0 \,\,\,\, \text{ in } \,\,\R^n_+.
	$$
\end{proposition}

To prove this monotonicity result, we use a moving planes argument, and for this reason we need a maximum principle in ``narrow'' sets for odd functions with respect to a hyperplane (see Proposition~\ref{Prop:MaxPrpNarrowOdd}). Recall that for a set $\Omega \subset \R^n$, we define the quantity $R(\Omega)$ as the smallest positive $R$ for which

\begin{equation}
\label{Eq:DefNarrow}
\dfrac{|B_R(x)\setminus \Omega|}{|B_R(x)|}\geq \dfrac{1}{2} \quad \text{ for every } x \in \Omega.
\end{equation}
If no such radius exists, we define $R(\Omega) = +\infty$. We say that a set $\Omega$ is ``narrow'' if $R(\Omega)$ is small depending on certain quantities.

An important result needed to establish the maximum principle in ``narrow'' sets is the following ABP-type estimate. It is proved in \cite{QuaasXia} for the fractional Laplacian, following the arguments in \cite{Cabre-ABP} (see also \cite{Cabre-Topics}). The proof for a general operator $L_K$ does not differ significantly from the one for the fractional Laplacian. Nevertheless, we include it here for the sake of completeness.

\begin{theorem}
	\label{Th:ABPEstimate}
	Let $\Omega \subset \R^n$ with $R(\Omega) < +\infty$. Let $L_K \in \lcal_0(n,\s,\lambda, \Lambda)$ and let $v\in L^1_\s(\R^n)\cap C^{\alpha}(\Omega)$, with $\alpha > 2\s$, such that $\sup_{\Omega} v < +\infty$ and satisfying
	$$
	\beqc{\PDEsystem}
	L_K v - c(x)v &\leq & h & \text{ in } \Omega\,, \\
	v & \leq & 0 & \text{ in } \R^n\setminus \Omega\,,
	\eeqc
	$$
	with $c(x)\leq 0$ in $\Omega$ and $h\in L^\infty(\Omega)$.
	
	Then,
	$$
	\sup_\Omega v \leq C R(\Omega)^{2\s} \norm{h}_{L^{\infty}(\Omega)}\,,
	$$
	where $C$ is a constant depending on $n$, $\s$, and $\Lambda$.
\end{theorem}

The only ingredient needed to show Theorem~\ref{Th:ABPEstimate} is the following weak Harnack inequality proved in  \cite{Cozzi-DeGiorgiClassesShort}.

\begin{proposition}[see Corollary 4.4 of \cite{Cozzi-DeGiorgiClassesShort}]
	
	\label{Prop:WeakHarnack}
	
	Let $\Omega \subset \R^n$ and $L_K\in (n,\s,\lambda, \Lambda)$. Let $w \in L^1_\s(\R^n)\cap C^{\alpha}(\Omega)$, with $\alpha > 2\s$, such that $w\geq 0$ in $\R^n$. Assume that $w$ satisfies weakly $L_K w \geq h$ in $\Omega$, for some $h\in L^\infty (\Omega)$. Then, there exists an exponent $\varepsilon > 0 $  and a constant $C > 1$, both depending on $n$, $\s$ and $\Lambda$, such that
	$$
	\bpar{ \fint_{B_{R/2}(x_0)} w^\varepsilon \d x}^{1/\varepsilon} \leq C \bpar{\inf_{B_{R}(x_0)} w + R^{2\s} \norm{h}_{L^{\infty}(\Omega)} }
	$$
	for every $x_0\in \Omega$ and $0<R<\dist(x_0, \partial \Omega)$.
\end{proposition}

With the previous weak Harnack inequality we can now establish the ABP estimate.

\begin{proof}[Proof of Theorem~\ref{Th:ABPEstimate}]
	First, note that it is enough to show it for $v > 0$ in $\Omega$ satisfying
	$$
	\beqc{\PDEsystem}
	L_K v &\leq & h & \text{ in } \Omega\,, \\
	v & \leq & 0 & \text{ in } \R^n\setminus \Omega\,.
	\eeqc
	$$
	Indeed, if we consider $\Omega_0 = \{x \in \Omega \ : \ v > 0\}$, then since $c \leq 0$ we have $L_K v \leq L_K v - c(x)v \leq h$ in $\Omega_0$.

	Define	$M:= \sup_\Omega v$. Then, for every $\delta >0$ there exists a point $x_\delta \in \Omega$ such that $v(x_\delta) \geq M - \delta$. Consider now the function $w := M - v^+$. Note that $0 \leq w \leq M$, $w(x_\delta) \leq \delta$, and $w \equiv M$ in $\R^n \setminus \Omega$. If we extend $h$ to be $0$ outside $\Omega$, we can easily verify that $L_K w \geq -h$ in $B_R(x_\delta)$.
	
	Now, by choosing $R= 2R(\Omega)$, and using the weak Harnack inequality of Proposition~\ref{Prop:WeakHarnack}, we get
	\begin{align*}
	M \left (\dfrac{1}{2}\right)^{1/\varepsilon} & \leq \bpar{M^{\varepsilon}\dfrac{|B_{R/2}(x_\delta)\setminus \Omega|}{|B_{R/2}(x_\delta)|}}^{1/\varepsilon}= \bpar{\dfrac{1}{|B_{R/2}(x_\delta)|} \int_{B_{R/2}(x_\delta)\setminus \Omega} w^\varepsilon \d x}^{1/\varepsilon} \\
	& \leq \bpar{ \fint_{B_{R/2}(x_\delta)} w^\varepsilon \d x}^{1/\varepsilon} \leq C \bpar{\inf_{B_{R}(x_\delta)} w + R^{2\s} \norm{h}_{L^{\infty}(\Omega)} } \\
	& \leq C \bpar{\delta + R^{2\s} \norm{h}_{L^{\infty}(\Omega)} }\,.
	\end{align*}
	The conclusion follows from letting $\delta \to 0$.
\end{proof}

As a consequence of this result, one can deduce easily a general maximum principle in ``narrow'' sets.

\begin{corollary}
	\label{Cor:MaxPpleNarrowDomains}
	Let $\Omega \subset \R^n$ with $R(\Omega) < +\infty$. Let $L_K \in \lcal_0(n,\s,\lambda, \Lambda)$ and let $v\in L^1_\s(\R^n)\cap C^{\alpha}(\Omega)$, with $\alpha > 2\s$, such that $\sup_{\Omega} v < +\infty$ and satisfying
	$$
	\beqc{\PDEsystem}
	L_K v + c(x)v &\leq & 0 & \text{ in } \Omega\,, \\
	v & \leq & 0 & \text{ in } \R^n\setminus \Omega\,,
	\eeqc
	$$
	with $c(x)$ bounded by below.
	
	Then, there exists a number $\overline{R} > 0$ such that $v \leq 0$ in $\Omega$ whenever $R(\Omega)< \overline{R}$.	
\end{corollary}

\begin{proof}
	We write $c= c^+ - c^-$, and therefore $L_K v -(-c^+)v \leq c^- v^+	$. By Theorem~\ref{Th:ABPEstimate} we get
	$$
	\sup_\Omega v \leq C R(\Omega)^{2\s} \norm{c^- v^+}_{L^\infty(\Omega)} \leq C R(\Omega)^{2\s} \norm{c^-}_{L^\infty(\Omega)} \sup_\Omega v\,.
	$$
	Hence, if $C R(\Omega)^{2\s} \norm{c^-}_{L^\infty(\Omega)}  <1 $, we deduce that $v\leq 0$ in $\Omega$.
\end{proof}

The previous maximum principle in ``narrow'' sets is not suitable enough to apply the moving planes method, and we need to adapt it to the setting of odd functions with respect to a hyperplane (see Proposition~\ref{Prop:MaxPrpNarrowOdd} below, which will be deduced  from Corollary~\ref{Cor:MaxPpleNarrowDomains}). The reason why we need it is the following. 
In the moving the argument, we would want to use a maximum principle in a ``narrow'' band and applied to an odd function with respect to a hyperplane. 
However, odd functions cannot have a constant sign in the exterior of a band, and in the hypotheses of Corollary~\ref{Cor:MaxPpleNarrowDomains} there is a prescribed constant sign of a function outside the set $\Omega$. 
Thus, we need another version of a maximum principle in ``narrow'' sets that applies to odd functions and only requires a constant sign of the function at one side of a hyperplane (in the spirit of the maximum principles of Proposition~\ref{Prop:MaximumPrincipleForOddFunctions}). 
This is accomplished with the following result.

\begin{proposition}
	\label{Prop:MaxPrpNarrowOdd}
	Let $H$ be a half-space in $\R^n$, and denote by $x^\#$ the reflection of any point $x$ with respect to the hyperplane $\partial H$. Let $L_K \in \lcal_0$ with a positive kernel $K$ satisfying
	\begin{equation}
	\label{Eq:KernelSymmetry}
	K(x-y) \geq K(x-y^\#), \,\,\,\,\text{for all } \,\, x,y\in H.
	\end{equation}
	Assume that $v\in L^1_\s(\R^n)\cap C^{\beta}(\Omega)$, with $\beta > 2\s$, satisfies
	$$
	\beqc{\PDEsystem}
	L_K  v &\geq& c(x)\,v  &\textrm{ in } \Omega\subset H,\\
	v &\geq& 0 &\textrm{ in } H\setminus\Omega,\\
	v(x) &=& -v(x^\#) &\textrm{ in } \R^n,
	\eeqc
	$$
	with $c(x)$ bounded below.
	
	Then, there exist a number $\overline{R}$ such that $v \geq 0$ in $H$ whenever $R(\Omega) \leq \overline{R}$.
\end{proposition}

\begin{proof}
	Let us begin by defining $\Omega_- = \{x\in \Omega \,:\,\, v<0\}$. We shall prove that $\Omega_-$ is empty. Assume by contradiction that it is not empty. Then, we split $ v = v_1+v_2$, where
	\begin{equation*}
	v_1(x) =
	\begin{cases}
	v(x)  &\textrm{ in } \Omega_-,\\
	0 &\textrm{ in } \R^n\setminus\Omega_-,
	\end{cases}
	\quad \text{ and } \quad
	v_2(x) =
	\begin{cases}
	0  &\textrm{ in } \Omega_-,\\
	v(x) &\textrm{ in } \R^n\setminus\Omega_-.
	\end{cases}
	\end{equation*}
	
	We first show that $L_K v_2\leq 0$ in $\Omega_-$. To see this, take $x\in\Omega_-$ and thus
	$$
	L_K v_2(x) = \int_{\R^n\setminus\Omega_-} -v_2(y)K(x-y) \d y = -\int_{\R^n\setminus\Omega_-} v(y)K(x-y) \d y.
	$$
	Now, we split $\R^n\setminus\Omega_-$ into
	$$
	A_1 = \Omega_-^\#,\,\,\,\,\,\,\,\text{ and }\,\,\,\,\,\,\, A_2 = \left(H\setminus\Omega_-\right)\cup\left(H\setminus\Omega_-\right)^\#,
	$$
	and we compute the previous integral in these two sets separately using that $v$ is odd. On the one hand, since $v\leq 0$ in $\Omega_-$ and $K\geq 0$ in $\R^n$, we have
	\begin{align*}
	-\int_{A_1} v(y)K(x-y) \d y = -\int_{\Omega_-} v(y^\#)K(x-y^\#) \d y  = \int_{\Omega_-} v(y)K(x-y^\#) \d y \leq 0.
	\end{align*}
	On the other hand, by the kernel inequality \eqref{Eq:KernelSymmetry}
	\begin{align*}
	-\int_{A_2} v(y)K(x-y) \d y = -\int_{H\setminus\Omega_-} v(y)K(x-y) \d y  -\int_{H\setminus\Omega_-} v(y^\#)K(x-y^\#) \d y \\
	= -\int_{H\setminus\Omega_-} v(y)\left(K(x-y)-K(x-y^\#)\right) \d y \leq 0.
	\end{align*}
	Thus, we get $L_K v_2 \leq 0$ in $\Omega_-$.
	
	Finally, since $L_K v_2 \leq 0$ in $\Omega_-$, it holds
	$$ 
	L_K v_1 = L_K v-L_K v_2 \geq L_K v \geq c(x)\,v = c(x)\,v_1 \,\,\,\,\text{ in }\,\,\Omega_-. 
	$$
	Therefore $v_1$ solves
	\begin{equation*}
	\beqc{\PDEsystem}
	L_K v_1 &\geq& c(x)\,v_1   &\textrm{ in } \,\Omega_-,\\
	v_1 &=& 0 &\textrm{ in }\,\R^n\setminus\Omega_-,
	\eeqc
	\end{equation*}
	and we can apply the usual maximum principle for ``narrow'' sets (Corollary~\ref{Cor:MaxPpleNarrowDomains}) to $v_1$ in $\Omega_-$.  We deduce that $v_1\geq 0$ in all $\R^n$ whenever $R(\Omega)\leq \overline{R}$. This contradicts the definition of $v_1$ since we assumed that $\Omega_-$ was not empty. Thus, $\Omega_- = \varnothing$ and this yields $v\geq 0$ in $\Omega$.
\end{proof}

\begin{remark}
	A maximum principle such as Proposition~\ref{Prop:MaxPrpNarrowOdd} was already proved for the fractional Laplacian in \cite{ChenLiLi}, but with the additional hypothesis that either $\Omega$ is bounded or $\liminf_{x\in  \Omega,\ |x|\to \infty} v(x) \geq 0$. In the proof of Theorem~3.1 in \cite{QuaasXia}, Quaas and Xia use a suitable argument (the truncation used in the previous proof, previously used by Felmer and Wang in \cite{FelmerWang}) to avoid the requirement of such additional hypotheses on $\Omega$ or $v$.
\end{remark}

With the maximum principle in ``narrow'' sets for odd functions with respect to a hyperplane we can use the moving plane argument. Now we establish Proposition~\ref{Prop:MonotonyHalfSpace}.

\begin{proof}[Proof of Proposition~\ref{Prop:MonotonyHalfSpace}]
	The proof is based on the moving planes method, and is exactly the same as the analogue proof of Theorem~3.1 in \cite{QuaasXia}, where Quaas and Xia establish an equivalent result for the fractional Laplacian. For this reason, we give here just a sketch. As usual, for $\lambda > 0$ we define $w_\lambda (x) = v(x_H,2\lambda - x_n)-v(x_H,x_n)$ (recall that $x_H\in \R^{n-1}$) and since the nonlinearity is Lipschitz, $w_\lambda$ solves, in both cases ---\eqref{Eq:P1} or \eqref{Eq:P2}---, the following problem:
	$$
	\beqc{\PDEsystem} 
	L_K  w_\lambda &=& c_\lambda(x)\,w_\lambda  &\textrm{ in } \Sigma_\lambda\subset H_\lambda,\\ 
	w_\lambda &\geq& 0 &\textrm{ in } H_\lambda\setminus\Sigma_\lambda,\\ 
	w_\lambda(x_H,2\lambda - x_n) &=& - w_\lambda(x_H,x_n)  &\textrm{ in } \R^n, 
	\eeqc 
	$$
	where $\Sigma_\lambda := \left\{ x = (x_H,x_n) \ : \ 0<x_n<\lambda \right\}$ and $H_\lambda := \left\{ x = (x_H,x_n) \ : \ x_n<\lambda \right\}$ and $c_\lambda$ is a bounded function. Note that $w_\lambda$ is odd with respect to $\partial H_\lambda$. Then, using the maximum principle in ``narrow'' sets for odd functions (Proposition~\ref{Prop:MaxPrpNarrowOdd}) we deduce that, if $\lambda$ is small enough, $w_\lambda>0$ in $\Sigma_\lambda$. 
	
	To conclude the proof, we define
	$$
	\lambda^* := \sup\{\lambda \ : \ w_\eta>0 \,\, \text{ in } \,\, \Sigma_\lambda \,\, \text{ for all } \,\, \eta<\lambda\}.
	$$
	Note that $\lambda^*$ is well defined (but may be infinite) by the previous argument. To conclude the proof, one has to show that $\lambda^*=\infty$. This can be done by proving that, if $\lambda^*$ is finite, then there exists a small $\delta_0 > 0$ such that for every $\delta \in (0,\delta_0]$ we have
	$$
	w_{\lambda^* +  \delta} (x) > 0 \quad \text{ in } \Sigma_{\lambda^*-\varepsilon}\setminus \Sigma_{\varepsilon}
	$$
	for some small $\varepsilon$. This can be established using a compactness argument exactly as in Lemma~3.1 of \cite{QuaasXia} and thus we omit the details. In the argument a Harnack inequality is needed, one can use for instance Proposition~\ref{Prop:HarnackSemilinear}. Finally, by the maximum principle in ``narrow'' sets we deduce that $w_{\lambda^* +  \delta} (x) > 0 $ in $\Sigma_{\lambda^*+\delta}$ if $\delta$ is small enough, contradicting the definition of $\lambda^*$.
\end{proof}

Now, we present the other important ingredient needed in the proof of Theorem~\ref{Th:SymmHalfSpace}. It is the following symmetry result.

\begin{proposition}
	\label{Prop:HalfSpaceLimUnif}
	Let $L_K \in \lcal_0$ and let $v$ be a bounded solution to one of the following problems:
	\begin{equation}
	\tag{P3}
	\label{Eq:P3}
	\beqc{\PDEsystem}
	L_K  v &=& f(v)  &\textrm{ in }\R^n\,,\\
	\ds \lim_{x_n \to \pm \infty} v(x_H,x_n) &=& \pm 1 \,\,\, &\textrm{ uniformly}.
	\eeqc
	\end{equation}
	
	\begin{equation}
	\tag{P4}
	\label{Eq:P4}
	\beqc{\PDEsystem}
	L_K  v &=& f(v)  &\textrm{ in }\R^n_+ = \{ x_n>0\} \,,\\
	v &=& 0  &\textrm{ in } \R^n \setminus \R^n_+ = \{ x_n\leq 0\}\,,\\
	\ds \lim_{x_n \to + \infty} v(x_H,x_n) &=& 1 \,\,\, &\textrm{ uniformly}.
	\eeqc
	\end{equation}
	\reqnomode
	Assume that there exists a $\delta > 0$ such that
	$$ f' \leq 0 \quad \text{ in } \quad [-1,-1+\delta]\cup[1-\delta,1], $$
	for problem \eqref{Eq:P3} and
	$$ f' \leq 0 \quad \text{ in } \quad [1-\delta,1] $$
	for problem \eqref{Eq:P4}.
	
	Then, $v$ depends only on $x_n$ and is increasing in that direction.
\end{proposition}

\begin{proof}
	It is based on the sliding method, exactly as in the proof of Theorem~1 in \cite{BerestyckiHamelMonneau}. The idea is, as usual, to define $ v^\tau(x) := v(x+\nu \tau) $ for every $\nu\in\R^n$ with $|\nu|=1$ and $\nu_n>0$, and the aim is to show that $v^\tau(x)-v(x)\geq 0$ for all $\tau\geq 0$. Despite the fact that $L_K$ is a nonlocal operator, the proof is exactly the same as the one in \cite{BerestyckiHamelMonneau} ---it only relies on the maximum principle, the translation invariance of the operator and the Liouville type result of Theorem~\ref{Th:LiouvilleSemilinearWholeSpace}. Therefore, we do not include here the details.
\end{proof}

Finally, we can proceed with the proof of Theorem~\ref{Th:SymmHalfSpace}.

\begin{proof}[Proof of Theorem~\ref{Th:SymmHalfSpace}]
	Note that by Proposition~\ref{Prop:HalfSpaceLimUnif} we only need to prove that
	$$
	\ds \lim_{x_n\to+ \infty} v(x_H,x_n) = 1
	$$
	uniformly. Therefore we divide the proof in two steps: first, we prove that the limit exists and is $1$, and then we prove that it is uniform.

	\textbf{Step 1:} Given $x_H\in \R^{n-1}$, then  $\ds \lim_{x_n\to +\infty} v(x_H,x_n) = 1$.
	
	By Proposition \ref{Prop:MonotonyHalfSpace} we know that $v$ is strictly increasing in the direction $x_n$. Since $v$ is also bounded by hypothesis, we know that, given $x_H\in\R^{n-1}$, the one variable function $v(x_H,\cdot)$ has a limit as $x_n\to +\infty$, which we call $\overline{v}(x_H)$. Note that, since $v(x_H,0) = 0$ and $v_{x_n}>0$, it follows that $\overline{v}(x_H) > 0$.
	
	Let $x_n^k$ be any increasing sequence tending to infinity. Define $v_k(x_H,x_n) := v(x_H,x_n+x_n^k)$. By the regularity theory of the operator $L_K $ (see Section~\ref{Sec:Preliminaries}) and a standard compactness argument, we see that, up to a subsequence, $v_k$ converge uniformly on compact sets to a function $v_\infty$ which is a classical solution to
	\begin{equation}
	\label{Eq:ProofSymmHalf-SemilinearEqWholeSpace}
	\beqc{\PDEsystem}
	L_K v_\infty &=& f(v_\infty)   &\textrm{ in } \,\R^n,\\
	v_\infty &\geq& 0   &\textrm{ in } \,\R^n.
	\eeqc
	\end{equation}
	By Theorem \ref{Th:LiouvilleSemilinearWholeSpace}, either $v_\infty\equiv 0$ or $v_\infty \equiv 1$. But, by construction,
	$$ v_\infty(x_H,0) = \lim_{k\to + \infty} v_k(x_H,0) = \lim_{k\to + \infty} v(x_H,x_n^k) = \overline{v}(x_H) > 0, $$
	and therefore the only possibility is
	$$ \lim_{x_n\to \infty} v(x_H,x_n) = 1 \quad \text{ for all } \ x_H\in\R^{n-1}. $$
	
	\textbf{Step 2:} The limit is uniform in $x_H$.
	
	Let us proceed by contradiction. Suppose that the limit is not uniform. This means that given any $\varepsilon>0$ small enough, there exists a sequence of points $(x_H^k,x_n^k)$ with $x_n^k\to +\infty$ such that $v(x_H^k,x_n^k) = 1-\varepsilon$. Similarly as before, the sequence of functions $\tilde{v}_k(x_H,x_n) = v(x_H+x_H^k,x_n+x_n^k)$ converge uniformly on compact sets to a function $\tilde{v}_\infty$ that also solves \eqref{Eq:ProofSymmHalf-SemilinearEqWholeSpace}. By Theorem \ref{Th:LiouvilleSemilinearWholeSpace}, either $\tilde{v}_\infty\equiv 0$ or $\tilde{v}_\infty \equiv 1$. But, by construction
	$$ 
	\tilde{v}_\infty(0,0) = \lim_{k\to +\infty} \tilde{v}_k(0,0) = \lim_{k\to +\infty} v(x_H^k,x_n^k) = 1-\varepsilon, 
	$$
	which is a contradiction for $\varepsilon>0$ small enough. Thus, the limit is uniform.
	
	Finally, by applying Proposition~\ref{Prop:HalfSpaceLimUnif}, we get that $v$ depends only on $x_n$ and is increasing in that direction.
\end{proof}

%%%%%%%%%%%%%%%%%%%%%%%%%%%%%
\section{Asymptotic behavior of a saddle-shaped solution}
\label{Sec:Asymptotic}
%%%%%%%%%%%%%%%%%%%%%%%%%%%%%

In this section, we show Theorem~\ref{Th:AsymptoticBehaviorSaddleSolution}, concerning the asymptotic behavior of the saddle-shaped solution. 

In order to establish the result, it is important to study one-dimensional layer solutions in $\R^n$. Actually, in relation with the available results concerning a conjecture by De Giorgi, in low dimensions all layer solutions are one-dimensional (see Subsection~\ref{Subsec:DeGiorgi}).

One-dimensional layer solutions in $\R^n$ are in correspondence with the ones in $\R$. This comes for free when dealing with the local case, since if $v$ is a solution to $-\ddot{v} = f(v)$ in $\R$,  then $w(x) = v(x\cdot e)$ solves $-\Delta w = f(w)$ in $\R^n$ for every unitary vector $e\in \R^n$. The same fact also happens for the fractional Laplacian, that is, if $v$ is a solution to $\fraclaplacian v = f(v)$ in $\R$, then $w(x)= v(x\cdot e)$ solves the same equation in $\R^n$. We can easily see this relation via the local extension problem.

Nevertheless, for a general operator $L_K$ this is not true anymore and we need a way to relate a solution to a one-dimensional problem with a one-dimensional solution to a $n$-dimensional problem. This is given in the next result. Some of its points appear in \cite{CozziPassalacqua} with a different notation but we state and prove them here for completeness.

\begin{proposition}
	\label{Prop:KernelsDimension}
	Let $L_K \in \mathcal{L}_0 (n,\s,\lambda,\Lambda)$ be a symmetric and translation invariant integro-differential operator of the form \eqref{Eq:DefOfLu} with kernel $K:\R^n\setminus \{0\} \to (0,+\infty) $. Define the one dimensional kernel $K_1 : \R \setminus \{0\} \to (0,+\infty) $ by
	\begin{equation}
	\label{Eq:OneDimKernel}
	K_1(\tau) := \int_{\R^{n-1}} K\left(\theta,\tau\right) \d \theta = |\tau|^{n-1} \int_{\R^{n-1}} K\left(\tau\sigma,\tau\right) \d \sigma.
	\end{equation}
	\begin{enumerate}[label=(\roman{*})]
		\item Let $v:\R\to\R$ and consider $w:\R^n\to\R$ defined by $w(x) = v(x_n)$. Then, $L_K w(x) = L_{K_1} v(x_n)$. If we assume moreover that $K$ is radially symmetric, then the same happens with $w(x) = v(x\cdot e)$ for every unitary vector $e\in \Sph^{n-1}$. That is, $L_K w(x) = L_{K_1} v(x \cdot e)$.
		\item  If $K$ is nonincreasing/decreasing in  the $x_n$-direction in $\{x_n>0\}$, then $K_1(\tau)$ is nonincreasing/decreasing in $(0,+\infty)$.
		\item $L_{K_1} \in \mathcal{L}_0 (1,\s,\lambda,\Lambda)$, and moreover, if $L_K $ is the fractional Laplacian in dimension $n$, then $L_{K_1}$ is the fractional Laplacian in dimension $1$.
		
	\end{enumerate}
\end{proposition}

\begin{proof}
	We start proving point $(i)$. We write $y=(y_H,y_n)$, with $y_H\in \R^{n-1}$.
	\begin{align*}
	L_K w(x) &= \int_{\R^n} \big( w(x)-w(y) \big) K(x-y) \d y \\
	&=\int_{\R^n} \big( v(x_n)-v(y_n) \big) K\left(x_H-y_H,x_n-y_n\right) \d y_H \d y_n.
	\end{align*}
	Now we make the change of variables $\theta = x_H-y_H$. That is,
	\begin{align*}
	L_K w(x) 	&= \int_{\R} \big( v(x_n)-v(y_n) \big) \int_{\R^{n-1}} K\left(\theta,x_n-y_n\right) \d \theta \d y_n \\
	&= \int_{\R} \big( v(x_n)-v(y_n) \big) K_1(x_n-y_n ) \d y_n = L_{K_1}v(x_n).
	\end{align*}
	This shows the first equality in \eqref{Eq:OneDimKernel}. The alternative expression of the kernel $K_1$, that is useful in some cases, can be obtained from the change of variables $\theta = \tau\sigma$. Furthermore, in the case of $K$ radially symmetric, the result is valid for $u(x) = v(x\cdot e)$ for every unitary vector $e\in \Sph^{n-1}$ after a change of variables in the previous computations.
	
	The proof of point $(ii)$ follows directly from the first expression of the unidimensional kernel $K_1$. That is,
	$$ 
	K_1(\tau_2)-K_1(\tau_1) = \int_{\R^{n-1}} \big( K(\theta,\tau_2) - K(\theta,\tau_1)\big) \d \theta \geq 0 \quad \text{ for any } \quad \tau_2>\tau_1>0. 
	$$
	
	We establish now point $(iii)$. To do it, we bound the kernel $K_1$ using the ellipticity condition on $K$:
	\begin{align*}
	K_1(\tau) &= |\tau|^{n-1} \int_{\R^{n-1}} K\left(\tau(\sigma,1)\right) \d\sigma \geq |\tau|^{n-1} \int_{\R^n} c_{n,\s} \frac{\lambda}{|\tau|^{n+2\s}(|\sigma|^2+1)^{\frac{n+2s}{2}}} \d\sigma \\
	&= c_{n,\s} \frac{\lambda}{|\tau|^{1+2\s}} \int_{\R^{n-1}} \frac{\d\sigma}{(|\sigma|^2+1)^{\frac{n+2\s}{2}}} = c_{n,\s} \frac{\lambda}{|\tau|^{1+2\s}} \dfrac{2 \pi^{\frac{n-1}{2}}}{\Gamma(\frac{n-1}{2})} \int_0^\infty \dfrac{r^{n-2}}{(r^2+1)^{\frac{n+2\s}{2}}} \d r \\	
	& = c_{n,\s} \frac{\lambda}{|t|^{1+2\s}} 
	\dfrac{\pi^{\frac{n-1}{2}} \Gamma(\frac{1}{2}+\s)}{\Gamma(\frac{n}{2}+\s)} 
	= c_{n,\s} \frac{\lambda}{|t|^{1+2\s}} \frac{c_{1,\s}}{c_{n,\s}} = c_{1,\s} \frac{\lambda}{|t|^{1+2\s}},
	\end{align*}
	where we have used the explicit value of the normalizing constant for the fractional Laplacian,
	\begin{equation}
	\label{Eq:ConstantFracLaplacian}
	c_{n,\s} = \s\dfrac{2^{2\s} \Gamma(\frac{n}{2}+\s) }{\pi^{n/2} \Gamma(1-\s)},
	\end{equation}
	and the definition of the Beta and Gamma functions. The upper bound for $K_1$ is obtained in the same way. Note that the previous computation is an equality with $\lambda = 1$ in the case of the fractional Laplacian.
\end{proof}

In the proof of Theorem~\ref{Th:AsymptoticBehaviorSaddleSolution} we will use some properties of the layer solution $u_0$, defined in \eqref{Eq:LayerSolution}.
First, in \cite{CozziPassalacqua} it is proved that there exists a constant $C$ such that
\begin{equation}
\label{Eq:PropertiesLayer}
|u_0 (x)-\sign(x)| \leq C |x|^{-2\s}  \quad \text{ and } \quad |\dot{u}_0 (x)| \leq C |x|^{-1-2\s}  \quad \text{ for large }|x|.
\end{equation}
In our arguments we need also to show that the second derivative of the layer goes to zero at infinity. This is the first statement of the following lemma.

\begin{lemma}
	\label{Lemma:SecondDerivativeLayer}
	Let $K_1:\R \setminus \{0\} \to (0,+\infty)$ be a symmetric kernel satisfying \eqref{Eq:Ellipticity} and assume that it is decreasing in $(0,+\infty)$. Let $u_0$ be the layer solution associated to the kernel $K_1$, that is, $u_0$ solving \eqref{Eq:LayerSolution}. Then, 
	\begin{enumerate}[label=(\roman{*})]
		\item $\ddot{u}_0 (x) \to 0$ as $x\to \pm \infty$.	
		\item  $\ddot{u}_0 (x) < 0$ in $(0,+\infty)$.
	\end{enumerate}
\end{lemma}

We prove here the first statement of this lemma, and we postpone the proof of the second one until the next section, since we need to use a maximum principle for the linearized operator $L_{K_1} - f'(u_0)$.

\begin{proof}[Proof of point (i) of Lemma~\ref{Lemma:SecondDerivativeLayer}]
	By contradiction, suppose that there exists an unbounded sequence $\{x_j\}$ satisfying $|\ddot{u}_0(x_j)|>\varepsilon$ for some $\varepsilon>0$. Note that by the symmetry of $u_0$ we may assume that $x_j\to + \infty$. Now define $w_j (x) := \ddot{u}_0(x+x_j)$. By differentiating twice the equation of the layer solution, we see that $\ddot{u}_0$ solves
	$$
	L_{K_1} \ddot{u}_0 = f''(u_0)\dot{u}_0^2 + f'(u_0)\ddot{u}_0 \quad \text{ in }\R.
	$$
	Hence, as $x_j \to +\infty$ a standard compactness argument combined with the asymptotic behavior given by \eqref{Eq:PropertiesLayer} yields that $w_j$ converges on compact sets to a function $w$ that solves
	$$
	L_{K_1}  w = f'(1)w \quad \text{ in }\R.
	$$
	In addition, since $|\ddot{u}_0(x_j)|>\varepsilon$ we have $|w(0)|\geq \varepsilon$.
	
	At this point we use Lemma~4.3 of \cite{CozziPassalacqua} to deduce that, since $f'(1)<1$, then $w\to 0$ as $|x| \to +\infty$. Therefore, if $w$ is not identically zero, it has either a positive maximum or a negative minimum, but this contradicts the maximum principle (recall that $f'(1)<1$). We conclude that $w\equiv0$ in $\R$, but this is a contradiction with $|w(0)|\geq \varepsilon$.
\end{proof}

Now we have all the ingredients to establish the asymptotic behavior of the saddle-solution. The proof follows exactly the same compactness arguments used to prove the analogous result in the local case (see \cite{CabreTerraII}) and for the fractional Laplacian using the extension problem (see \cite{Cinti-Saddle, Cinti-Saddle2}). Thus we will omit some details. The main ingredients too establish this results are the translation invariance of the operator, the Liouville type and symmetry results of Theorems~\ref{Th:LiouvilleSemilinearWholeSpace} and \ref{Th:SymmHalfSpace} and a stability argument (recall the comments in Section~\ref{Sec:Preliminaries}).

\begin{proof}[Proof of Theorem~\ref{Th:AsymptoticBehaviorSaddleSolution}]
	By contradiction, assume that the result does not hold. Then, there exists an $\varepsilon>0$ and an unbounded sequence $\{x_k\}$, such that
	\begin{equation}
	\label{Eq:ContradictionAsymptotic}
	|u(x_k)-U(x_k)|+|\nabla u(x_k)-\nabla U(x_k)|+|D^2u(x_k)-D^2U(x_k)| > \varepsilon.
	\end{equation}
	By the symmetry of $u$, we may assume without loss of generality that $x_k \in \overline{\ocal}$, and by continuity we can further assume $ x_k \notin \ccal$. 
	
	Let $d_k:=\dist(x_k,\ccal)$. We distinguish two cases:
	
	\textbf{Case 1: $\{d_k\}$ is an unbounded sequence.} In this situation, we may assume that $d_k \geq 2k$. Define
	$$
	w_k(x) := u(x+x_k), 
	$$
	which satisfies $0<w_k<1$ in $\overline{B_k}$ and
	$$
	L_K w_k = f(w_k) \ \textrm{ in } B_k.
	$$
	Letting $k\to +\infty$, by standard estimates for the operators of the class $\lcal_0$ (see Section~\ref{Sec:Preliminaries}) and the Arzelà-Ascoli theorem, we have that, up to a subsequence, $w_k$ converges on compact sets to a function $w$ which is a pointwise solution to
	$$
	\beqc{\PDEsystem}
	L_K  w &=& f(w) & \textrm{ in }\R^n\,,\\
	w &\geq& 0 & \textrm{ in } \R^n\,.
	\eeqc
	$$
	
	Then, by Theorem~\ref{Th:LiouvilleSemilinearWholeSpace}, either $w\equiv 0$ or $w\equiv 1$. First, note that $w$ cannot be zero. Indeed, since $w_k$ are stable with respect to perturbations supported in $B_k$ (see the comments in Section~\ref{Sec:Preliminaries} and Remark~\ref{Remark:uStableinO}), $w$ is stable in $\R^n$, which means that the linearized operator $L_K-f'(w)$ is a positive operator. Nevertheless, if $w\equiv 0$, then the linearized operator $L_K-f'(w) = L_K-f'(0)$ is negative for sufficiently large balls, since $f'(0)>0$ and the first eigenvalue of $L_K$ is of order $R^{-2\s}$ in balls of radius $R$ (as in Lemma~\ref{Lemma:FirstOddEigenfunction}, see Proposition~9 of \cite{ServadeiValdinoci}). Therefore $w\equiv 1$. 
	
	On the other hand, since $d_k\rightarrow +\infty$ and $U(x_k) =  u_0(d_k)$, we get by the properties of the layer solution that $U(x_k) \rightarrow 1$, $\nabla U(x_k) \rightarrow 0$ and $D^2U(x_k) \rightarrow 0$ ---see \eqref{Eq:PropertiesLayer} and Lemma~\ref{Lemma:SecondDerivativeLayer}. From this and condition \eqref{Eq:ContradictionAsymptotic} we get
	$$
	|u(x_k)-1|+|\nabla u(x_k)|+|D^2u(x_k)| > \varepsilon/2,
	$$
	for $k$ big enough. This yields that 
	$$
	|w_k(0)-1|+|\nabla w_k(0)|+|D^2w_k(0)| > \varepsilon/2,
	$$
	and this contradicts $w \equiv 1$. 
	
	\textbf{Case 2: $\{d_k\}$ is a bounded sequence.}
	In this situation, at least for a subsequence, we have that $d_k \rightarrow d$. Now, for each $x_k$ we define $x_k^0$ as its projection on $\ccal$. Therefore, we have that $ \nu_k^0 := (x_k-x_k^0)/d_k$ is the unit normal to $\ccal$. Through a subsequence, $ \nu_k^0 \rightarrow \nu$ with $|\nu|=1$.
	
	We define
	$$ w_k (x) := u(x+x_k^0), $$
	which solves
	$$ L_K  w_k = f(w_k) \ \text{in } \R^n. $$
	Similarly as before, by letting $k\to +\infty$, up to a subsequence $w_k$ converges on compact sets to a function $w$ which is a pointwise solution to
	$$
	\beqc{\PDEsystem}
	L_K  w &=& f(w)  &\textrm{ in } H:=\{x\cdot \nu >0\}\,,\\
	w &\geq& 0  &\textrm{ in } H\,,\\
	\,\,w \text{ is odd with respect to } H. \span\span\span \,
	\eeqc
	$$
	For the details about the fact that $\ocal+x_k^0 \rightarrow H$, see \cite{CabreTerraII}.
	
	As in the previous case, by stability $w$ cannot be zero, and thus $w>0$ in $H$ (by the strong maximum principle for odd functions with respect to a hyperplane, see \cite{ChenLiLi}). Hence, by Theorem \ref{Th:SymmHalfSpace}, $w$ only depends on $x\cdot \nu$ and is increasing. Finally, by the uniqueness of the layer solution, $w(x) = u_0(x\cdot \nu)$ and
	\begin{align*}
	u(x_k) &= w_k(x_k-x_k^0) = w(x_k-x_k^0) + \mathrm{o}(1) \\
	&= u_0((x_k-x_k^0)\cdot \nu) + \mathrm{o}(1) = u_0((x_k-x_k^0)\cdot \nu_k^0) + \mathrm{o}(1) \\
	&= u_0(d_k |\nu_k^0|^2) + \mathrm{o}(1) = u_0(d_k) + \mathrm{o}(1) = U (x_k) + \mathrm{o}(1),
	\end{align*}
	contradicting \eqref{Eq:ContradictionAsymptotic}. The same is done for $\nabla u$ and $D^2 u$.
\end{proof}

\begin{remark}
	\label{Remark:u>delta}
	The previous result yields that, for $\varepsilon>0$ the saddle-shaped solution satisfies $u\geq\delta$ in the set $\ocal_\varepsilon := \{(x',x'')\in \R^m\times\R^m \ : \ |x''|+\varepsilon <|x'| \}$, for some positive constant $\delta$. That is, thanks to the asymptotic result, and since $U(x)\geq u_0(\varepsilon/\sqrt{2})$ for $x\in \ocal_\varepsilon$, there exists a radius $R>0$ such that $u(x)\geq U(x)/2\geq u_0(\varepsilon/\sqrt{2})/2$ if $x\in \ocal_\varepsilon\setminus B_R$. Moreover, since $u$ is positive in the compact set $\overline{\ocal_\varepsilon}\cap\overline{B_R}$ it has a positive minimum in this set, say $m>0$. Therefore, if we choose $\delta = \min\{m,u_0(\varepsilon/\sqrt{2})/2\}$ we obtain the desired result.
\end{remark}

%%%%%%%%%%%%%%%%%%%%%%%%%%%%%
\section{Maximum principles for the linearized operator}
%%%%%%%%%%%%%%%%%%%%%%%%%%%%%%
\label{Sec:MaximumPrinciple}

In this section we show that the linearized operator $L_K  -f'(u)$ satisfies the maximum principle in $\ocal$. This, combined with the asymptotic result of Theorem~\ref{Th:AsymptoticBehaviorSaddleSolution}, yields the uniqueness of the saddle-shaped solution.

In order to prove the maximum principle of Proposition~\ref{Prop:MaximumPrincipleLinearized}, we need a maximum principle in ``narrow'' sets, stated next.

\begin{proposition}
	\label{Prop:MaximumPrincipleNarrowDomainsOdd}
	Let $\varepsilon>0$ and let
	$$
	\ncal_\varepsilon \subset \{(x',x'')\in \R^m\times\R^m \ : \ |x''|<|x'|<|x''|+ \varepsilon\} \subset \ocal
	$$ 
	be an open set (not necessarily bounded).  
	Let $K$ be a radially symmetric kernel satisfying the positivity condition \eqref{Eq:KernelInequality} and such that $L_K\in \lcal_0$. Let $v\in C(\overline{\ncal_\varepsilon})\cap C^\alpha(\ncal_\varepsilon)\cap L^1_\s(\R^{2m})$, for some $\alpha > 2\s$, be a doubly radial function satisfying
	\begin{equation}
	\label{Eq:AssumptionsMaxPNarrow}
	\beqc{\PDEsystem}
	L_K v + c(x)v&\leq & 0 &\textrm{ in } \ncal_\varepsilon\,,\\
	v &\leq & 0 &\textrm{ in } \ocal \setminus \ncal_\varepsilon\,,\\
	- v(x^\star) & = & v(x) &\textrm{ in } \R^{2m},\\
	\ds \limsup_{x\in \ncal_\varepsilon, \ |x|\to \infty} v(x) &\leq & 0\,,
	\eeqc
	\end{equation}
	with $c$ a function bounded by below.
	
	Under these assumptions there exists $\overline{\varepsilon}>0$ depending only on $\lambda, m, \s$ and $\norm{c_-}_{L^\infty}$ such that, if $\varepsilon<\overline{\varepsilon}$, then $v \leq 0$ in $\ncal_\varepsilon$. 
\end{proposition}

\begin{proof}
	Assume, by contradiction, that
	$$
	M := \sup_{\ncal_\varepsilon} v > 0\,.
	$$
	Under the assumptions \eqref{Eq:AssumptionsMaxPNarrow}, $M$ must be attained at an interior point $x_0 \in \ncal_\varepsilon$. Then,
	\begin{equation}
	\label{Eq:InequalitiesMaxPNarrowProof}
	0 \geq L_K  v(x_0) + c(x_0)v(x_0) \geq L_K  v(x_0) - \norm{c_-}_{L^\infty(\ncal_\varepsilon)}M\,.
	\end{equation} 
	Now, we compute $L_K  v(x_0)$. Since $v$ is doubly radial and odd with respect to the Simons cone, we can use the expression \eqref{Eq:OperatorOddF} to write
	\begin{align*}
	L_K v(x_0) &= \int_{\ocal} \big (M - v(y) \big) \big (\overline{K}(x_0,y) -\overline{K}(x_0,y^\star)\big) \d y + 2M\int_{\ocal} \overline{K}(x_0,y^\star)\d y\\
	&\geq2M \int_{\ocal} \overline{K}(x_0,y^\star)\d y,
	\end{align*}
	where the inequality follows from being $M$ the supremum of $v$ in $\ocal$ and the kernel inequality \eqref{Eq:KernelInequality}. Combining this last inequality with \eqref{Eq:InequalitiesMaxPNarrowProof}, we obtain
	$$
	0 \geq L_K  v(x_0) + c(x_0)v(x_0)  \geq M \left(  2 \int_{\ocal} \overline{K}(x_0,y^\star)\d y - \norm{c_-}_{L^\infty(\ncal_\varepsilon)}
	\right)\,.
	$$
	
	Finally, if we use the lower bound of \eqref{Eq:ZeroOrderTerm} and the fact that $\dist(x_0,\ccal) \leq \varepsilon/\sqrt{2}$, we get
	\begin{align*}
	0 &\geq M \left( 2 \int_{\ocal} \overline{K}(x_0,y^\star)\d y - \norm{c_-}_{L^\infty(\ncal_\varepsilon)}
	\right ) \geq M \left(\frac{1}{C}\dist(x_0,\ccal)^{-2\s}-\norm{c_-}_{L^\infty(\ncal_\varepsilon)}\right) \\ &\geq M \left(\frac{1}{C}\varepsilon^{-2\s}-\norm{c_-}_{L^\infty(\ncal_\varepsilon)}\right).
	\end{align*}
	Therefore, for $\varepsilon$ small enough, we arrive at a contradiction that follows from assuming that the supremum is positive.
\end{proof}

\begin{remark}
	Using same arguments as in the proof of Proposition~\ref{Prop:MaxPrpNarrowOdd}, the previous result can be extended to general doubly radial ``narrow'' sets (that is, assuming that the set $\ncal_\varepsilon$ in the statement of 	Proposition~\ref{Prop:MaximumPrincipleNarrowDomainsOdd} satisfies \eqref{Eq:DefNarrow}, instead of just being contained in an $\varepsilon$-neighborhood of the cone). 
	Indeed, we only need to replace the symmetry with respect to a hyperplane by the symmetry with respect to the Simons cone and use the kernel inequality \eqref{Eq:KernelInequality} ---note that in this case, the assumption at infinity in \eqref{Eq:AssumptionsMaxPNarrow} is not needed. 
	Nevertheless, we preferred to present the result for sets that are contained in an $\varepsilon$-neighborhood of the Simons cone, since we are only going to use the maximum principle in such sets. 
	In addition, the crucial fact that the sets are contained in $\{(x',x'')\in \R^m\times\R^m \ : \ |x''|<|x'|<|x''|+ \varepsilon\}$ makes the argument rather simple.
\end{remark}

Once this maximum principle in ``narrow'' sets is available, we can proceed with the proof of Proposition~\ref{Prop:MaximumPrincipleLinearized}.

\begin{proof}[Proof of Proposition~\ref{Prop:MaximumPrincipleLinearized}]

	For the sake of simplicity, we will denote 
	$$
	\mathscr{L} w := L_K w - f'(u)w - cw\,.
	$$
	A crucial point in this proof is that $u$ is a positive supersolution of the operator $\mathscr{L}$. Indeed, since $f$ is strictly concave in $(0,1)$ and $f(0)=0$, then $f'(\tau)\tau<f(\tau)$ for all $\tau>0$, and thus 
	\begin{equation}
	\label{Eq:uSupersolLinearized}
	\mathscr{L} u = L_K u - f'(u)u - cu \geq f(u) - f'(u)u > 0 \quad \textrm{ in } \Omega \subset \ocal\,,
	\end{equation}
	where in the first inequality we have used that $u>0$ in $\ocal$ and that $c\leq 0$.

	By contradiction, assume that there exists $x_0\in \Omega$ such that $v(x_0)> 0$. We will show next that, if we assume this, we deduce $v\leq 0$ in $\Omega$, arriving at a contradiction.

	Let $\varepsilon > 0$ be such that the maximum principle of Proposition~\ref{Prop:MaximumPrincipleNarrowDomainsOdd} is valid and define the following sets:
	$$
	\Omega_\varepsilon := \Omega \cap \{|x'| > |x''| + \varepsilon\}\quad \textrm{ and } \quad 
	\ncal_\varepsilon := \Omega \cap \{|x''| < |x'| < |x''| + \varepsilon\}\,.
	$$
	Define also, for $\tau \geq 0$, 
	$$
	w := v - \tau u.
	$$
	
	First, we claim that $w\leq 0$ in $\Omega$ if $\tau$ is big enough. To see this, note first that by the asymptotic behavior of the saddle-shaped solution, we have 
	\begin{equation}
	\label{Eq:u>delta}
	u \geq \delta > 0 \quad \textrm{ in } \overline{\Omega}_\varepsilon\,,
	\end{equation}
	for some $\delta >0$ (see Remark~\ref{Remark:u>delta}). Therefore, $w < 0$ in $\overline{\Omega}_\varepsilon$ if $\tau$ is big enough. Moreover, since $v\leq 0$ in $\ocal\setminus\Omega$, we have 
	$$
	w \leq 0 \quad \textrm{ in } \ocal \setminus \ncal_\varepsilon\,.
	$$
	Furthermore, it also holds
	$$
	\limsup_{x\in \ncal_\varepsilon, \ |x|\to \infty} w(x) \leq 0
	$$
	and, by \eqref{Eq:uSupersolLinearized},
	$$
	\mathscr{L} w = \mathscr{L} v - \tau \mathscr{L} u \leq 0 \textrm{ in } \ncal_\varepsilon\,.
	$$
	Thus, since $w$ is odd with respect to $\ccal$, we can apply Proposition~\ref{Prop:MaximumPrincipleNarrowDomainsOdd} in $\ncal_\varepsilon$ to deduce that
	$$
	w \leq 0 \quad \textrm{ in } \Omega\,,
	$$
	if $\tau$ is big enough.
	
	Now, define 
	$$
	\tau_0:= \inf \setcond{\tau > 0}{v - \tau u \leq 0 \ \textrm{ in } \Omega}.
	$$
	By the previous claim, $\tau_0$ is well defined. Moreover, it is easy to see that $\tau_0 > 0$. Indeed, it is obvious $v - \tau_0 u \leq 0 $ in $\Omega$ and thus, since $v(x_0)>0$, we have $-\tau_0 u(x_0) < v(x_0) - \tau_0 u (x_0) \leq 0$. Using that $u(x_0)>0$, it follows that  $\tau_0 > 0$.
	
	We claim that $v - \tau_0 u \not \equiv 0$. Indeed, if $v - \tau_0 u \equiv 0$ then $v = \tau_0 u$ and thus, by using \eqref{Eq:uSupersolLinearized}, the equation for $v$, and the fact that $\tau_0 > 0$, we get 
	$$
	0 \geq \mathscr{L} v(x_0) = \tau_0 \mathscr{L} u(x_0) > 0\,, 
	$$
	which is a contradiction.
	
	Then, since $v - \tau_0 u \not \equiv 0$, the strong maximum principle for odd functions (see Proposition~\ref{Prop:MaximumPrincipleForOddFunctions}) yields
	$$
	v - \tau_0 u < 0 \quad \textrm{ in }\Omega\,.
	$$
	Therefore, by continuity, the assumption on $v$ at infinity and \eqref{Eq:u>delta}, there exists $0 < \eta <\tau_0$ such that 
	$$
	\tilde{w} := v - (\tau_0 - \eta) u < 0 \quad \textrm{ in }\overline{\Omega}_\varepsilon\,.
	$$
	Note that here we used crucially \eqref{Eq:u>delta}, and this is the reason for which we needed to introduce the sets $\Omega_\varepsilon$ and $\ncal_\varepsilon$. Using again the maximum principle in ``narrow'' sets with $\tilde{w}$ in $\ncal_\varepsilon$, we deduce that 
	$$
	v - (\tau_0 - \eta) u \leq 0 \quad \textrm{ in }\Omega\,,
	$$
	and this contradicts the definition of $\tau_0$. Hence, $v\leq 0$ in $\Omega$ and, as we said, this contradicts our initial assumption on the existence of a point $x_0$ where $v(x_0)>0$.
\end{proof}

Note that if in the previous result we assume that $\partial\Omega \cap \ccal$ is empty, then $\Omega$ is at a positive distance to the cone and the lower bound on $u$ in \eqref{Eq:u>delta} holds in $\Omega$. In this case no maximum principle in ``narrow'' sets is required in the previous argument. Instead, if we want to consider sets with $\partial\Omega \cap \ccal \neq \varnothing$, we need to introduce the set $\Omega_\varepsilon$ to have the uniform lower bound \eqref{Eq:u>delta} and be able to carry out the proof.

The same argument used in the previous proof can be used to establish the remaining statement of Lemma~\ref{Lemma:SecondDerivativeLayer}.

\begin{proof}[Proof of point (ii) of Lemma~\ref{Lemma:SecondDerivativeLayer}]
	
	Let $v = \ddot{u}_0$. First we show that $v\leq 0$ in $(0,+\infty)$. To see this, note that since $f$ is concave and by point (i) of Lemma~\ref{Lemma:SecondDerivativeLayer}, we have that
	$$
	\beqc{\PDEsystem}
	L_{K_1} v - f'(u_0)v &\leq &0 & \text{ in } (0,+\infty)\,.\\
	v(x) &= &-v(-x) & \text{ for every } x\in \R\,,\\
	\ds \limsup_{x\to +\infty} v(x) &= & 0\,.
	\eeqc
	$$
	Now, we follow the proof of Proposition~\ref{Prop:MaximumPrincipleLinearized} but with the previous problem, replacing $u$ by $u_0$ and using that
	$$
	L_{K_1} u_0 - f'(u_0)u_0 > 0 \quad \text{ in } (0,+\infty)\,. 
	$$
	All the arguments are the same, using the maximum principle of Proposition~\ref{Prop:MaxPrpNarrowOdd} in the set $(0,\varepsilon)$, and yield that $v\leq 0$ in $(0,+\infty)$.
	
	The fact that $\ddot{u}_0 = v < 0$ in $(0,+\infty)$ can be readily deduced from the strong maximum principle for odd functions in $\R$, as follows. Suppose by contradiction that there exists a point $x_0\in (0,+\infty)$ such that $v(x_0) = 0$. Then,
	\begin{align*}
	0 &\geq L_{K_1} v (x_0) = - \int_{-\infty}^{+\infty}	v(y) K_1(x_0 - y) \d y \\
	&= - \int_{-\infty}^{+\infty} v(y) \big( K_1(x_0 - y) - K_1(x_0 + y)\big) \d y > 0\,,
	\end{align*}
	arriving at a contradiction. Here we have used that $v\not \equiv 0$ and the fact that $K_1$ is decreasing in $(0,+\infty)$, which yields $K_1(x - y) \geq  K_1(x + y)$ for every $x>0$ and $y>0$.
\end{proof}

%%%%%%%%%%%%%%%%%%%%%%%%%%%%%%%%%%%%%%%%%%%%%%%%%%%%%%%%%%%%%%%%%%%%%%%%%%%%
%%%%%%%%%%%%%%%%%%%%%%%%%%%%%%%%%%%%%%%%%%%%%%%%%%%%%%%%%%%%%%%%%%%%%%%%%%%%
\section*{Acknowledgements}

The authors thank Xavier Cabré for his guidance and useful discussions on the topic of this paper. They also thank Matteo Cozzi for interesting discussions on some parts of this article.

%%%%%%%%%%%%%%%%%%%%%%%%%%%%%%%%%%%%%%%%%%%%%%%%%%%%%%%%%%%%%%%%%%%%%%%%%%%%
%%%%%%%%%%%%%%%%%%%%%%%%%%%%%%%%%%%%%%%%%%%%%%%%%%%%%%%%%%%%%%%%%%%%%%%%%%%%
\bibliographystyle{amsplain}
\bibliography{biblio}

\providecommand{\bysame}{\leavevmode\hbox to3em{\hrulefill}\thinspace}
\providecommand{\MR}{\relax\ifhmode\unskip\space\fi MR }
\providecommand{\MRhref}[2]{%
  \href{http://www.ams.org/mathscinet-getitem?mr=#1}{#2}
}
\providecommand{\href}[2]{#2}
\begin{thebibliography}{10}

\bibitem{AlbertiBouchitteSeppecher}
G.~Alberti, G.~Bouchitt\'e, and P.~Seppecher, \emph{{Phase transition with the
  line-tension effect}}, Arch. Rational Mech. Anal. \textbf{144} (1998), 1--46.

\bibitem{BarriosEtAl-Monotonicity}
B.~Barrios, L.~Del~Pezzo, J.~Garc\'ia-Meli\'an, and A.~Quaas,
  \emph{{Monotonicity of solutions for some nonlocal elliptic problems in
  half-spaces}}, Calc. Var. Partial Differential Equations \textbf{56} (2017),
  Art. 39.

\bibitem{BarriosEtAl-Symmetry}
\bysame, \emph{{Symmetry results in the half-space for a semi-linear fractional
  {L}aplace equation}}, Ann. Mat. Pura Appl. (4) \textbf{197} (2018),
  1385--1416.

\bibitem{BarriosPeralSoriaValdinoci}
B.~Barrios, I.~Peral, F.~Soria, and E.~Valdinoci, \emph{{A {W}idder's type
  theorem for the heat equation with nonlocal diffusion}}, Arch. Ration. Mech.
  Anal. \textbf{213} (2014), 629--650.

\bibitem{BerestyckiHamelMonneau}
H.~Berestycki, F.~Hamel, and R.~Monneau, \emph{{One-dimensional symmetry of
  bounded entire solutions of some elliptic equations}}, Duke Math. J.
  \textbf{103} (2000), 375--396.

\bibitem{BerestyckiHamelNadi}
H.~Berestycki, F.~Hamel, and N.~Nadirashvili, \emph{{The speed of propagation
  for KPP type problems II: General domains}}, J. Amer. Math. Soc. \textbf{23}
  (2010), 1--34.

\bibitem{BucurValdinoci}
C.~Bucur and E.~Valdinoci, \emph{{Nonlocal Diffusion and Applications}},
  Lecture Notes of the Unione Matematica Italiana, Springer International
  Publishing, 2016.

\bibitem{Cabre-ABP}
X.~Cabr\'e, \emph{{On the Alexandroff-Bakel’man-Pucci estimate and the
  reversed Hölder inequality for solutions of elliptic and parabolic
  equations.}}, Comm. Pure Appl. Math. \textbf{48} (1995), 539--570.

\bibitem{Cabre-Topics}
\bysame, \emph{{Topics in regularity and qualitative properties of solutions of
  nonlinear elliptic equations}}, Discrete Contin. Dyn. Syst. \textbf{8}
  (2002), 331--359.

\bibitem{Cabre-Saddle}
\bysame, \emph{{Uniqueness and stability of saddle-shaped solutions to the
  Allen-Cahn equation}}, J. Math. Pures Appl. \textbf{98} (2012), 239--256.

\bibitem{CabreSireII}
X.~Cabr\'e and Y.~Sire, \emph{{Nonlinear equations for fractional Laplacians
  II: Existence, uniqueness, and qualitative properties of solutions}}, Trans.
  Amer. Math. Soc. \textbf{367} (2015), 911--941.

\bibitem{CabreSolaMorales}
X.~Cabr\'e and J.~Sol\`a-Morales, \emph{{Layer solutions in a half-space for
  boundary reactions}}, Comm. Pure Appl. Math. \textbf{58} (2005), 1678--1732.

\bibitem{CabreTerraII}
X.~Cabr\'e and J.~Terra, \emph{{Qualitative properties of saddle-shaped
  solutions to bistable diffusion equations}}, Comm. Partial Differential
  Equations \textbf{35} (2010), 1923--1957.

\bibitem{ChenLiLi}
W.~Chen, C.~Li, and Y.~Li, \emph{{A direct method of moving planes for the
  fractional {L}aplacian}}, Adv. Math. \textbf{308} (2017), 404--437.

\bibitem{ChenLiZhang}
W.~Chen, Y.~Li, and R.~Zhang, \emph{{A direct method of moving spheres on
  fractional order equations}}, J. Funct. Anal. \textbf{272} (2017),
  4131--4157.

\bibitem{Cinti-Saddle}
E.~Cinti, \emph{{Saddle-shaped solutions of bistable elliptic equations
  involving the half-Laplacian}}, Ann. Sc. Norm. Super. Pisa Cl. Sci. (5)
  \textbf{12} (2013), 623--664.

\bibitem{Cinti-Saddle2}
\bysame, \emph{{Saddle-shaped solutions for the fractional Allen-Cahn
  equation}}, Discrete Contin. Dyn. Syst. Ser. S \textbf{11} (2018), 441--463.

\bibitem{Cozzi-DeGiorgiClassesLong}
M.~Cozzi, \emph{{Regularity results and Harnack inequalities for minimizers and
  solutions of nonlocal problems: A unified approach via fractional De Giorgi
  classes}}, J. Funct. Anal. (2017), no.~11, 4762--4837.

\bibitem{Cozzi-DeGiorgiClassesShort}
\bysame, \emph{{Fractional De Giorgi classes and applications to nonlocal
  regularity theory}}, Contemporary research in elliptic {PDE}s and related
  topics, Springer INdAM Ser., vol.~33, Springer, Cham, 2019, pp.~277--299.

\bibitem{CozziFigalli-Survey}
M.~Cozzi and A.~Figalli, \emph{Regularity theory for local and nonlocal minimal
  surfaces: an overview}, Lecture Notes in Math., vol. 2186, Springer, Cham,
  2017, pp.~117--158.

\bibitem{CozziPassalacqua}
M.~Cozzi and T.~Passalacqua, \emph{{One-dimensional solutions of non-local
  Allen-Cahn-type equations with rough kernels}}, J. Differential Equations
  \textbf{260} (2016), 6638--6696.

\bibitem{DaviladelPinoWei}
J.~D{\'a}vila, M.~del Pino, and J.~Wei, \emph{{Nonlocal $s$-minimal surfaces
  and Lawson cones}}, J. Differential Geom. \textbf{109} (2018), 111--175.

\bibitem{delPinoKowalczykWei}
M.~del Pino, M.~Kowalczyk, and J.~Wei, \emph{{On De Giorgi's conjecture in
  dimension $N\geq9$}}, Ann. of Math. \textbf{174} (2011), 1485--1569.

\bibitem{DipierroSoaveValdinoci}
S.~Dipierro, N.~Soave, and E.~Valdinoci, \emph{On fractional elliptic equations
  in {L}ipschitz sets and epigraphs: regularity, monotonicity and rigidity
  results}, Math. Ann. \textbf{369} (2017), 1283--1326.

\bibitem{Evans}
L.~C. Evans, \emph{{Partial Differential Equations: Second Edition}}, 2nd ed.,
  Graduate Studies in Mathematics, AMS, 2010.

\bibitem{FallWethMonotonicity}
M.~Fall and T.~Weth, \emph{{Monotonicity and nonexistence results for some
  fractional elliptic problems in the half-space}}, Commun. Contemp. Math.
  \textbf{18} (2016), 1550012.

\bibitem{FarinaValdinoci-DeGiorgi}
A.~Farina and E.~Valdinoci, \emph{{The state of the art for a conjecture of De
  Giorgi and related problems}}, Recent progress on reaction-diffusion systems
  and viscosity solutions, World Sci. Publ., Hackensack, NJ, 2009, pp.~74--96.

\bibitem{Felipe-Sanz-Perela:SaddleFractional}
J.C. Felipe-Navarro and T.~Sanz-Perela, \emph{{Uniqueness and stability of the
  saddle-shaped solution to the fractional Allen-Cahn equation}}, to appear in
  Rev. Mat. Iberoam, preprint arXiv 1810.08483 (2018).

\bibitem{FelipeSanz-Perela:IntegroDifferentialI}
\bysame, \emph{{Semilinear integro-differential equations, I: odd solutions
  with respect to the Simons cone}}, J. Funct. Anal. \textbf{278} (2020),
  108309, 48.

\bibitem{FelmerWang}
P.~Felmer and Y.~Wang, \emph{{Radial symmetry of positive solutions to
  equations involving the fractional {L}aplacian}}, Commun. Contemp. Math.
  \textbf{16} (2014), 1350023, 24.

\bibitem{Gonzalez}
M.d.M Gonz\'alez, \emph{{Gamma convergence of an energy functional related to
  the fractional {L}aplacian}}, Calc. Var. Partial Differential Equations
  \textbf{36} (2009), 173--210.

\bibitem{HamelRosOtonSireValdinoci}
F.~Hamel, X.~Ros-Oton, Y.~Sire, and E.~Valdinoci, \emph{{A one-dimensional
  symmetry result for a class of nonlocal semilinear equations in the plane}},
  Ann. Institut H. Poincar\'e \textbf{34} (2017), 469--482.

\bibitem{JerisonMonneau}
D.~Jerison and R.~Monneau, \emph{{Towards a counter-example to a conjecture of
  De Giorgi in high dimensions}}, Ann. Mat. Pura Appl. (4) \textbf{183} (2004),
  439--467.

\bibitem{LiZhang}
Y.~Li and L.~Zhang, \emph{{Liouville-type theorems and harnack-type
  inequalities for semilinear elliptic equations}}, J. Anal. Math. \textbf{90}
  (2003), 27--87.

\bibitem{LiuWangWei-StabilitySaddle}
Y.~Liu, K.~Wang, and J.~Wei, \emph{{Stability of the saddle solutions for the
  Allen-Cahn equation}}, preprint arXiv 2001.07356 (2020).

\bibitem{QuaasXia}
A.~Quaas and A.~Xia, \emph{{Liouville type theorems for nonlinear elliptic
  equations and systems involving fractional {L}aplacian in the half space}},
  Calc. Var. Partial Differential Equations \textbf{52} (2015), 641--659.

\bibitem{RosOton-Survey}
X.~Ros-Oton, \emph{{Nonlocal elliptic equations in bounded domains: a survey}},
  Publ. Mat. \textbf{60} (2016), 3--26.

\bibitem{RosOtonSerra-Stable}
X.~Ros-Oton and J.~Serra, \emph{{Regularity theory for general stable
  operators}}, J. Differential Equations \textbf{260} (2016), 8675--8715.

\bibitem{SanzPerela-Thesis}
T.~Sanz-Perela, \emph{{Stable solutions of nonlinear fractional elliptic
  problems}}, Ph.D. thesis, Universitat Polit\`ecnica de Catalunya, 2019.

\bibitem{SavinValdinoci-GammaConvergence}
O.~Savin and E.~Valdinoci, \emph{{$\Gamma$-convergence for nonlocal phase
  transitions}}, Ann. Inst. H. Poincar\'e Anal. Non Lin\'eaire \textbf{29}
  (2012), 479--500.

\bibitem{SavinValdinoci-Cones}
\bysame, \emph{{Regularity of nonlocal minimal cones in dimension 2}}, Calc.
  Var. Partial Differential Equations \textbf{48} (2013), 33--39.

\bibitem{SerraC2s+alphaRegularity}
J.~Serra, \emph{{$C^{\sigma + \alpha}$ regularity for concave nonlocal fully
  nonlinear elliptic equations with rough kernels}}, Calc. Var. Partial
  Differential Equations \textbf{54} (2015), 3571--3601.

\bibitem{ServadeiValdinoci}
R.~Servadei and E.~Valdinoci, \emph{{Variational methods for non-local
  operators of elliptic type}}, Discrete Contin. Dyn. Syst. \textbf{33} (2013),
  2105--2137.

\bibitem{Valdinoci2013FractionalPerimeter}
E.~Valdinoci, \emph{A fractional framework for perimeters and phase
  transitions}, Milan Journal of Mathematics \textbf{81} (2013), 1--23.

\end{thebibliography}

\end{document}